\documentclass[11pt]{amsart}
\usepackage{amsmath,amssymb}

\allowdisplaybreaks
\theoremstyle{plain}
\newtheorem{theorem}{Theorem}[section]
\newtheorem{lemma}[theorem]{Lemma}

\newtheorem{corollary}[theorem]{Corollary}
\theoremstyle{definition}
\newtheorem{definition}[theorem]{Definition}
\newtheorem{notation}[theorem]{Notation}
\newtheorem{remark}[theorem]{Remark}

\begin{document}
\title[Radical factorization]{Radical factorization in commutative rings, monoids and multiplicative lattices}

\author[Bruce Olberding]{Bruce Olberding}
\address{Department of Mathematical Sciences, New Mexico State University, Las Cruces, NM 88003, United States}
\email{bruce@nmsu.edu}

\author[Andreas Reinhart]{Andreas Reinhart}
\address{Department of Mathematical Sciences, New Mexico State University, Las Cruces, NM 88003, United States}
\email{andreas.reinhart@uni-graz.at}

\thanks{The research of the second-named author of this work was supported by the Austrian Science Fund FWF, Project Number J4023-N35}

\begin{abstract}
In this paper we study the concept of radical factorization in the context of abstract ideal theory in order to obtain a unified approach to the theory of factorization into radical ideals and elements in the literature of commutative rings, monoids and ideal systems. Using this approach we derive new characterizations of classes of rings whose ideals are a product of radical ideals, and we obtain also similar characterizations for classes of ideal systems in monoids and star ideals in integral domains.
\end{abstract}

\maketitle

\section{Introduction}

This article is concerned with the factorization of ideals in commutative rings and monoids into products of radical ideals. Much is known about the integral domains, rings and cancellative monoids whose ideals possess this factorization property; see \cite{AD,FHL,HOR,MC,Olb,R,V,Y} and their references. While by many measures, radical factorization is quite a bit weaker than prime factorization, it is still the case that
a ring or monoid whose ideals have radical factorization must meet a number of strong demands, as is evidenced in the characterizations in the cited references. However, one also finds factorization into radical ideals among special subclasses of ideals of rings and monoids. Rather than require all ideals to have the radical factorization property, we thus can consider restricted classes of ideals. This is analogous to the passage from Dedekind domains to Krull domains: The property that every proper ideal of a domain is a product of prime ideals characterizes Dedekind domains, and hence is rather restrictive. Taking a more flexible approach and working up to divisorial closure, we have the familiar property of Krull domains that divisorial ideals factor into prime ideals up to divisorial closure; i.e., every proper divisorial ideal $I$ is of the form $I=((P_1\cdots P_n)^{-1})^{-1}$ for some height $1$ prime ideals $P_1,\ldots,P_n$. Thus by working with a restricted class of ideals and a more flexible interpretation of product we find Dedekind factorization outside the class of Dedekind domains.

Our goal in this article is to show that the radical factorization property also can be found in more general settings by suitably restricting the ideals considered and having a more flexible notion of product. In fact, our methods allow us to work with both ideal systems of commutative rings as well as monoids. Rather than develop ad hoc approaches to each of these different settings, we give a unified treatment through the use of multiplicative lattices. The collections of ideals that we will be interested in (in both the ring and monoid settings) can be viewed in an obvious way as a lattice having a multiplicative structure. On a more philosophical level, this approach shows that the phenomenon of radical factorization, at least to the extent that we consider it here, is a consequence of the arithmetic of the ideals of the ring, monoid or ideal system, rather than the elements in these ideals, i.e., our analysis of these properties involves quantification over ideals rather than elements.
As we recall in Section 2, multiplicative lattices have been well studied by many authors, and so there are a number of tools available for our purposes.

Thus we develop first in Sections 2--6 a theory of radical factorization for multiplicative lattices and use the results obtained in this fashion to derive in Sections 7 and 8 a number of results and characterizations of radical factorization in commutative rings, monoids and ideal systems.

Throughout the paper we assume all rings, monoids and semigroups are commutative and have more than one element.

\section{Multiplicative lattices}

Our methods in this paper involve those of abstract ideal theory, and so our main tool is that of a multiplicative lattice.

\begin{definition}\label{Def 2.1} A {\it multiplicative lattice} is a partially ordered multiplicative monoid $(\mathcal{L},\leq)$ with the following properties.
\begin{enumerate}
\item[(a)] $(\mathcal{L},\leq)$ is a complete lattice, and hence has a top element $1$ and a bottom element $0$.
\item[(b)] $x(\bigvee_{y\in S}y)=\bigvee_{y\in S} xy$ for each $x\in\mathcal{L}$ and $S\subseteq\mathcal{L}$.
\item[(c)] The top element $1$ is the multiplicative identity of $\mathcal{L}$.
\end{enumerate}
If also $ab=0$ implies $a$ or $b$ is $0$ for all $a,b\in\mathcal{L}$, then $\mathcal{L}$ is a {\it multiplicative lattice domain.}
\end{definition}

An element $x$ of a lattice $\mathcal{L}$ is {\it compact} if whenever $x\leq\bigvee_\alpha y_\alpha$ for some collection $\{y_\alpha\}$ of elements of $\mathcal{L}$, we have that $x\leq y_{\alpha_1}\vee\cdots\vee y_{\alpha_n}$ for some $\alpha_1,\ldots,\alpha_n$. Let $\mathcal{L}^*$ denote the set of compact elements of $\mathcal{L}$.

\begin{definition}\label{Def 2.2} A multiplicative lattice $\mathcal{L}$ is a {\it $C$-lattice} if the set $\mathcal{L}^*$ of compact elements is multiplicatively closed (i.e., $1\in\mathcal{L}^*$ and $xy\in\mathcal{L}^*$ for all $x,y\in\mathcal{L}^*$) and every element in $\mathcal{L}$ is a join of compact elements. The argument in \cite[Lemma 1]{JJA} shows that a multiplicative lattice $\mathcal{L}$ is a $C$-lattice if and only if there is some multiplicatively closed set $A\subseteq\mathcal{L}^*$ such that every element of $\mathcal{L}$ is a join of elements from $A$.
\end{definition}

\begin{notation}\label{Not 2.3} Let $\mathcal{L}$ be a multiplicative lattice, and let $x,y\in\mathcal{L}$. We use the following notation.
\begin{enumerate}
\item $(y:x)=\bigvee\{a\in\mathcal{L}\mid ax\leq y\}.$
\item $\sqrt{x}=\bigvee\{y\in\mathcal{L}\mid y^n\leq x$ for some $n\in\mathbb{N}\}$.
\item $\mbox{\rm Max}(\mathcal{L}) = $ the set of maximal elements of $\mathcal{L}$.
\end{enumerate}
\end{notation}

An element $x$ is {\it $\ell$-radical} if $x=\sqrt{x}$. In the literature on multiplicative lattices, an $\ell$-radical element is called simply a radical element, but we use the term $\ell$-radical since there exists the different notion of radical elements in monoids (i.e., an element of a monoid is called radical if the ideal generated by it is a radical ideal). To avoid similar confusion, we use the terms $\ell$-principal and $\ell$-invertible in the next definition in place of what are called principal and invertible elements in the context of multiplicative lattices.

One motivation for consideration of multiplicative lattices is that these structures capture fundamental properties of ideals of commutative rings. The notion of a principal element in a multiplicative lattice, first introduced by Dilworth (see \cite{Dil} for an overview of the history of this notion), plays a role similar to that of finitely generated locally principal ideals in commutative rings (see \cite{AJo} and \cite[Theorem 2]{McC} for more on this). Weaker versions of principality also prove useful since they encode familiar properties such as being a multiplicative or cancellative element (see~\cite{And}).

\begin{definition}\label{Def 2.4} Let $\mathcal{L}$ be a multiplicative lattice, and let $x\in\mathcal{L}$.
\begin{enumerate}
\item $x$ is {\it cancellative} if it is a cancellative element of the monoid (i.e., $xy=xz$ implies $y=z$ for all $y,z\in\mathcal{L}$ or equivalently $xy\leq xz$ implies $y\leq z$ for all $y,z\in\mathcal{L}$).
\item $x$ is {\it weak meet principal} if $x\wedge y=(y:x)x$ for all $y\in\mathcal{L}$.
\item $x$ is {\it meet principal} if $y\wedge zx=\left((y:x)\wedge z\right)x$ for all $y,z\in\mathcal{L}$.
\item $x$ is {\it weak join principal} if $(xy:x)\leq y\vee (0:x)$ for all $y\in\mathcal{L}$.
\item $x$ is {\it join principal} if $y\vee (z:x)=((yx\vee z):x)$ for all $y,z\in\mathcal{L}$.
\item $x$ is {\it $\ell$-principal} if it is both meet and join principal.
\item $x$ is {\it $\ell$-invertible} if $x$ is $\ell$-principal and cancellative.
\end{enumerate}
The lattice $\mathcal{L}$ is {\it principally generated} if each element is a join of $\ell$-principal elements.
\end{definition}

Note that in a multiplicative lattice domain, every nonzero $\ell$-principal element is cancellative and hence is $\ell$-invertible.

The next lemma collects several useful properties of $\ell$-principal and $\ell$-invertible elements; see \cite[Lemma~2.3]{AJ} and \cite[Corollary 3.3]{Dil}.

\begin{lemma}\label{Lem 2.5}
Let $\mathcal{L}$ be a multiplicative lattice, and let $x,y\in\mathcal{L}$.
\begin{enumerate}
\item If $x,y$ are $\ell$-principal, then $xy$ is $\ell$-principal.
\item $xy$ is $\ell$-invertible if and only if $x$ and $y$ are $\ell$-invertible.
\item $1\in\mathcal{L}$ is $\ell$-invertible.
\end{enumerate}
\end{lemma}

As we recall next, the fact that a $C$-lattice has a good supply of compact elements allows for a localization theory that behaves like that of commutative rings.

\begin{notation}\label{Not 2.6} Let $\mathcal{L}$ be a $C$-lattice, and let $p\in\mathcal{L}$ be $\ell$-prime (i.e., $p\not=1$ and for all $a,b\in\mathcal{L}$, $ab\leq p$ implies $a\leq p$ or $b\leq p$). We use the terminology of $\ell$-prime elements to distinguish these elements from the prime elements in monoids. For each $x\in\mathcal{L}$, we set \begin{center} $x_p=\bigvee\{a\in\mathcal{L}^*\mid\exists b\in\mathcal{L}^*$ such that $b\not\leq p$ and $ab\leq x\},$\end{center}
and we let $\mathcal{L}_p=\{x_p\mid x\in\mathcal{L}\}$.
\end{notation}

\begin{lemma}\label{Lem 2.7}{\em (cf.~\cite[pp.~201--203]{JJSP})} Let $\mathcal{L}$ be a $C$-lattice, let $x,y\in\mathcal{L}$ and let $p\in\mathcal{L}$ be an $\ell$-prime element.
\begin{enumerate}
\item $x_p=1$ if and only if $x\not\leq p$.
\item $(xy)_p=(x_py_p)_p$.
\item $(x\wedge y)_p=x_p\wedge y_p$.
\item If $p\in\mbox{\rm Max}(\mathcal{L})$, then $(p^n)_p=p^n$ for all $n\in\mathbb{N}$.
\item $x=\bigwedge_{m\in\mbox{\rm Max}(\mathcal{L})} x_m$ and $x=y$ if and only if $x_m=y_m$ for all $m\in\mbox{\rm Max}(\mathcal{L})$.
\item If $x$ is compact, then $(y:x)_p=(y_p:x_p)$.
\item If $x$ is both weak meet and weak join principal, then $x$ is compact.
\item $\sqrt{x}=\bigwedge\{q\in\mathcal{L}\mid q $ is $\ell$-prime with $x\leq q\}$.
\item $\sqrt{x_p}=\sqrt{x}_p$.
\end{enumerate}
\end{lemma}

\begin{lemma}\label{Lem 2.8} Let $\mathcal{L}$ be a C-lattice, $x\in\mathcal{L}$ and $p\in\mathcal{L}$ an $\ell$-prime element which is minimal above $x$. Then $\sqrt{x_p}=p$.
\end{lemma}

\begin{proof} It is clear that $\sqrt{x_p}\leq p$. Let $a\in\mathcal{L}^*$ be such that $a\leq p$. Set $\Omega=\{ba^n\mid b\in\mathcal{L}^*,b\not\leq p,n\in\mathbb{N}_0\}$. It suffices to show that there is some $z\in\Omega$ such that $z\leq x$.

Assume that $z\not\leq x$ for each $z\in\Omega$. Since $\Omega$ is a multiplicatively closed set of compact elements of $\mathcal{L}$, there is some $\ell$-prime $q\in\mathcal{L}$ such that $x\leq q$ and $z\not\leq q$ for each $z\in\Omega$. Note that if $c\in\mathcal{L}^*$ is such that $c\not\leq p$, then $c\in\Omega$, and hence $c\not\leq q$. Therefore, $x\leq q\leq p$, and hence $a\leq p=q$. Since, $a\in\Omega$, we have that $a\not\leq q$, a contradiction.
\end{proof}

\begin{definition}\label{Def 2.9} If $\mathcal{L}$ is a multiplicative lattice and the length of the longest chain of $\ell$-prime elements is $n$, then the {\it dimension} of $\mathcal{L}$ is $n-1$.
\end{definition}

We will be mainly interested in zero-dimensional elements of $\mathcal{L}$, i.e., those elements $x$ for which the only $\ell$-prime elements above $x$ are maximal.

A lattice $\mathcal{L}$ is {\it modular} if for all $x,y,z\in\mathcal{L}$ such that $x\leq z$ it follows that $(x\vee y)\wedge z=x\vee (y\wedge z)$ (equivalently, for all $x,y,z\in\mathcal{L}$ such that $x\leq z$ we have that $(x\vee y)\wedge z\leq x\vee (y\wedge z)$.
In Section 8, whether a multiplicative lattice is modular is a key issue for determining the ideal systems to which our methods can be applied. The relevance of the modularity condition is due to the following lemma.

\begin{lemma}\label{Lem 2.10}{\em (cf.~\cite[Proposition 1.1]{And})} Let $\mathcal{L}$ be a multiplicative lattice and $x\in\mathcal{L}$ a cancellative element.
\begin{enumerate}
\item $x$ is weak join principal.
\item If $x$ is weak meet principal, then $x$ is meet principal.
\item If $\mathcal{L}$ is modular and $x$ is weak meet principal, then $x$ is $\ell$-principal.
\end{enumerate}
In particular, if $\mathcal{L}$ is modular, then an element of $\mathcal{L}$ is $\ell$-invertible if and only if it is weak meet principal and cancellative.
\end{lemma}

\begin{proof} (1) If $y\in\mathcal{L}$, then $(xy:x)=y\leq y\vee (0:x)$. Therefore, $x$ is weak join principal.

(2) Let $x$ be weak meet principal. First we show that $yx\wedge zx=(y\wedge z)x$ for each $y,z\in\mathcal{L}$. Let $y,z\in\mathcal{L}$. Since $x$ is weak meet principal and $yx\wedge zx\leq x$, there is some $a\in\mathcal{L}$ such that $yx\wedge zx=ax$. We infer that $ax\leq yx$ and $ax\leq zx$. Therefore, $a\leq y\wedge z$, and hence $xy\wedge xz\leq x(y\wedge z)\leq xy\wedge xz$.

Now let $y,z\in\mathcal{L}$. Then $y\wedge zx=(y\wedge x)\wedge zx=((y:x)x)\wedge zx=((y:x)\wedge z)x$.

(3) Let $\mathcal{L}$ be modular and let $x$ be weak meet principal. By (2) it remains to show that $x$ is join principal. Let $y,z\in\mathcal{L}$. Note that $yx\leq x$, and hence $(y\vee(z:x))x=yx\vee (z:x)x=yx\vee (z\wedge x)=(yx\vee z)\wedge x=((yx\vee z):x)x$, since $x$ is weak meet principal. Since $x$ is cancellative, we infer that $y\vee(z:x)=((yx\vee z):x)$.
\end{proof}

\section{Radical factorization in $C$-lattices}

The purpose of this section is to give a sufficient condition in Theorem~\ref{Thm 3.1} for a zero-dimensional element of a $C$-lattice to factor into a product of $\ell$-radical elements. An application of this to commutative rings is given in Theorem~\ref{Thm 7.8}. In the next section, we use Theorem~\ref{Thm 3.1} to find necessary and sufficient conditions for a lattice domain to have the property that every element is a product of $\ell$-radical elements.

\begin{theorem}\label{Thm 3.1} Let $\mathcal{L}$ be a $C$-lattice, and let $x \ne 1$ be a zero-dimensional element of $\mathcal{L}$. If each maximal element above $x$ is also above a zero-dimensional $\ell$-radical element that is compact and weak meet principal, then $x=y_1\cdots y_k$ for some $\ell$-radical elements $y_1\leq\cdots\leq y_k$.
\end{theorem}

\begin{proof} We prove the theorem by establishing a series of claims.

\medskip

{\noindent}{\textsc{Claim 1}}: $(x:\sqrt{x})_m=(x_m:\sqrt{x}_m)$ for all $m\in\mbox{\rm Max}(\mathcal{L})$.

\medskip

Let $m\in\mbox{\rm Max}(\mathcal{L})$. If $x\not\leq m$, then $\sqrt{x}\not\leq m$, and since $x\leq (x:\sqrt{x})$ we have that $(x:\sqrt{x})\not\leq m$. Therefore, by Lemma~\ref{Lem 2.7}(1), $$(x:\sqrt{x})_m=1=(x_m:\sqrt{x}_m).$$

Now suppose $x\leq m$. It is clear that $(x:\sqrt{x})_m\leq (x_m:\sqrt{x}_m)$. By assumption there is a zero-dimensional $\ell$-radical element $y\in\mathcal{L}$ such that $y\leq m$ and $y$ is weak meet principal and compact. Observe that $y_m=m=\sqrt{x}_m$ by Lemma~\ref{Lem 2.8}. Next we show that $(x\vee y)_n\geq\sqrt{x}_n$ for all $n\in\mbox{\rm Max}(\mathcal{L})$. Let $n\in\mbox{\rm Max}(\mathcal{L})$. If $x\leq n$, then $(x\vee y)_n\geq y_n\geq n=n_n\geq\sqrt{x}_n$. If $x\not\leq n$, then $(x\vee y)_n=1\geq\sqrt{x}_n$.

We infer that $x\vee y\geq\sqrt{x}$, and hence $(x:y)=(x:(x\vee y))\leq (x:\sqrt{x})$. This implies that $(x_m:\sqrt{x}_m)=(x_m:y_m)=(x:y)_m\leq (x:\sqrt{x})_m$.

\medskip

{\noindent}{\textsc{Claim 2}}: $x=\sqrt{x}x_1$ for some $x_1\in\mathcal{L}$.

\medskip

Using Lemma~\ref{Lem 2.7}(5), we may verify the equality locally by showing that if $m\in\mbox{\rm Max}(\mathcal{L})$, $$x_m=\left(\sqrt{x}(x:\sqrt{x})\right)_m.$$ Let $m\in\mbox{\rm Max}(\mathcal{L})$. Consider first the case that $x\leq m$. By assumption, there exists an $\ell$-radical zero-dimensional element $y\in\mathcal{L}$ such that $y\leq m$ and $y$ is compact and weak meet principal. Since $y$ is a weak meet principal element of $\mathcal{L}$, we have that $y_m$ is a weak meet principal element of $\mathcal{L}_m$. Moreover, it follows by Lemma~\ref{Lem 2.8} that $\sqrt{x}_m=m=y_m$, and thus $\sqrt{x}_m$ is a weak meet principal element of $\mathcal{L}_m$. Therefore, $x_m=(\sqrt{x}_m(x_m:\sqrt{x}_m))_m$, and hence $x_m=(\sqrt{x}(x:\sqrt{x}))_m$ by Claim 1.

On the other hand, if $m\in\mbox{\rm Max}(\mathcal{L})$ is not above $x$, then Claim 1 and the assumption that $x\not\leq m$ imply $$x_m= 1=\left(\sqrt{x}_m(x_m:\sqrt{x}_m)\right)_m=\left(\sqrt{x}_m(x:\sqrt{x})_m\right)_m=\left(\sqrt{x}(x:\sqrt{x})\right)_m.
$$ Thus $x_m=(\sqrt{x}(x:\sqrt{x}))_m$. Since we have shown this equality holds for all $m\in\mbox{\rm Max}(\mathcal{L})$, we conclude by Lemma~\ref{Lem 2.7}(5) that $x=\sqrt{x}(x:\sqrt{x})$.

\medskip

{\noindent}{\textsc{Claim 3}}: There exist a positive integer $t$ and zero-dimensional $\ell$-radical elements $z_1,\ldots,z_t$ such that $z_1\cdots z_t\leq x$.

\medskip

Let $\{m_\alpha\}$ denote the collection of maximal elements above $x$. By assumption, for each $\alpha$ there is a zero-dimensional $\ell$-radical element $y_\alpha\in\mathcal{L}$ such that $y_\alpha\leq m_\alpha$ and $y_\alpha$ is compact and weak meet principal. Consider the element $a=\bigvee_\alpha (\sqrt{x}:y_\alpha).$ We show $a=1$. For each $\alpha$, since $y_\alpha$ is compact we have by Lemma~\ref{Lem 2.7}(6) that $$a_{m_\alpha}\geq (\sqrt{x}:y_\alpha)_{m_\alpha}=(\sqrt{x}_{m_\alpha}:(y_\alpha)_{m_\alpha})=(m_\alpha:m_\alpha)=1.$$
Therefore, $a_{m_\alpha}=1$, which forces $a\not\leq m_\alpha$. Since $x\leq a$, the maximal elements above $a$ are among the $m_\alpha$. Therefore, $a=\bigwedge_\alpha a_{m_\alpha}=1$.

By assumption, the top element $1$ of $\mathcal{L}$ is compact, so there are $s\in\mathbb{N}$ and $\alpha_1,\ldots,\alpha_s$ such that $1=(\sqrt{x}:y_{\alpha_1})\vee\cdots\vee (\sqrt{x}:y_{\alpha_s})$. For each $1\leq i\leq s$, let $z_i=y_{\alpha_i}$. Observe that $1=(\sqrt{x}:z_1)\vee\cdots\vee (\sqrt{x}:z_s)\leq (\sqrt{x}:z_1\cdots z_s)$, and thus $1=(\sqrt{x}:z_1\cdots z_s)$. Therefore, $z_1\cdots z_s\leq\sqrt{x}$. Hence the fact that $z_1\cdots z_s$ is compact and below $\sqrt{x}$ implies there is $n\in\mathbb{N}$ such that $(z_1\cdots z_s)^n\leq x$.

\medskip

{\noindent}{\textsc{Claim 4}}: There exist $\ell$-radical elements $y_1,\ldots,y_k$ of $\mathcal{L}$ such that $y_1\leq\cdots\leq y_k$ and $x=y_1\cdots y_k$.

\medskip

By Claim 2, $x=\sqrt{x}x_1$ for some $x_1\in\mathcal{L}$. If $x_1\ne 1$, then since $x_1$ is zero-dimensional, we may apply Claim 2 to obtain $x=\sqrt{x}\sqrt{x_1}x_2$ for some $x_2\in\mathcal{L}$. Continuing in this fashion, we obtain that either $x=y_1\cdots y_k$ for some $\ell$-radical elements $y_1\leq y_2\leq\cdots\leq y_k$, in which case the proof is complete, or there are $\ell$-radical elements $y_1\leq y_2\leq\cdots $ such that for each $i\in\mathbb{N}$, there is $x_{i+1}\ne 1$ with $y_{i+1}=\sqrt{x_{i+1}}$ and $x=y_1y_2\cdots y_ix_{i+1}$.

Suppose the latter case holds. By Claim 3 there are zero-dimensional $\ell$-radical elements $z_1,\ldots,z_t\in\mathcal{L}$ such that $z_1\cdots z_t\leq x$. We claim that $x=y_1y_2\cdots y_{t+1}$.

We verify the equality $x=y_1 y_2\cdots y_{t+1}$ locally. Let $m\in\mbox{\rm Max}(\mathcal{L})$, and suppose first that $x_{t+1}\leq m$. Using the fact that $$y_1\leq\cdots\leq y_t\leq y_{t+1}=\sqrt{x_{t+1}}\leq m,$$ we have that $$z_1\cdots z_t\leq x=y_1\ldots y_tx_{t+1}\leq m^{t+1}.$$ Since $(z_i)_m\geq m$ for each $i$, localizing at $m$ yields $$m^t\leq\left((z_1)_m\cdots (z_t)_m\right)_m=(z_1\cdots z_t)_m\leq x_m\leq m^{t+1}.$$ Therefore,
$x_m=m^t=m^{t+1}$ for all $m\in\mbox{\rm Max}(\mathcal{L})$ with $x_{t+1}\leq m$. Since $y_{t+1}=\sqrt{x_{t+1}}\leq m$, we have that $(y_1\cdots y_ty_{t+1})_m=m^{t+1}=x_m.$ On the other hand, if $x_{t+1}\not\leq m$, then again since $y_{t+1}=\sqrt{x_{t+1}}$ we have that
\begin{eqnarray*} x_m\:=\: (y_1\cdots y_tx_{t+1})_m &=&\left((y_1)_m\cdots (y_t)_m(x_{t+1})_m\right)_m\\
&=&\left((y_1)_m\cdots (y_t)_m(y_{t+1})_m\right)_m\\
&=& (y_1\cdots y_{t+1})_m.
\end{eqnarray*}
Therefore, $x=y_1\cdots y_{t+1}$ since this equality holds locally.
\end{proof}

\begin{corollary}\label{Cor 3.2} Let $\mathcal{L}$ be $C$-lattice, and let $y$ be a compact zero-dimensional $\ell$-radical element in $\mathcal{L}$ that is weak meet principal. Then for each $x\in\mathcal{L}$ with $y\leq\sqrt{x}$, there are $\ell$-radical elements $y_1\leq\cdots\leq y_k$ such that $x=y_1\cdots y_k$.
\end{corollary}

\begin{proof} Apply Theorem~\ref{Thm 3.1}.
\end{proof}

\section{Radical factorization in lattice domains}

In this section we characterize the radical factorization property in $C$-lattices for which every element is a join of $\ell$-invertible elements. The characterization in Theorem~\ref{Thm 4.6}, the main result of this section, will serve as a basis for most of the applications given in later sections.

\begin{definition}\label{Def 4.1} A {\it radical factorial lattice} $\mathcal{L}$ is a multiplicative lattice $\mathcal{L}$ such that every element in $\mathcal{L}$ is a product of $\ell$-radical elements.
\end{definition}

The proof of Theorem~\ref{Thm 4.6} relies on three technical lemmas, all of which are motivated by arguments from \cite{AD}, although our proofs are more complicated due to the generality of our setting. We show in Section 7 how to derive some of the results from \cite{AD} in our context.

\begin{lemma}\label{Lem 4.2} Let $\mathcal{L}$ be a $C$-lattice, and let $p$ be an $\ell$-prime element of $\mathcal{L}$. If $x$ is a compact join principal element such that $x\not\leq p$ and $x^2\vee p$ is a product of $\ell$-radical elements, then for each $\ell$-prime element $q$ minimal over $x\vee p$, we have that $p\leq (q^2)_q$ and $q\ne (q^2)_q$.
\end{lemma}

\begin{proof} Let $q$ be an $\ell$-prime element minimal over $x\vee p$. By assumption, $x^2\vee p=x_1\cdots x_k$ for some $\ell$-radical elements $x_i$. Therefore, $(x^2\vee p)_q=((x_1)_q\cdots (x_k)_q)_q$. Since $q$ is minimal over $x\vee p$, and hence minimal over $x^2\vee p$, it follows that for each $i$, $(x_i)_q$ is either $1$ or $q$ is minimal over $(x_i)_q$. In the latter case, since $x_i$ is $\ell$-radical, $(x_i)_q=q$. Consequently, from the fact that $(x^2\vee p)_q=((x_1)_q\cdots (x_k)_q)_q$, we conclude that $(x^2\vee p)_q=(q^n)_q$ for some $n\in\mathbb{N}$.

Assume that $(x^2\vee p)_q=q$. Then $q=(x^2\vee p)_q\leq (x\vee p)_q\leq q$, and thus $x_q\leq (x\vee p)_q=(x^2\vee p)_q$. The fact that $x$ is join principal, $p$ is $\ell$-prime and $x\not\leq p$ implies that $((p\vee x^2):x)=x\vee (p:x)=x\vee p$. Since $x$ is compact, we have that $1=((p\vee x^2)_q: x_q)=((p\vee x^2): x)_q=(x\vee p)_q$ by Lemma~\ref{Lem 2.7}(6). We infer by Lemma~\ref{Lem 2.7}(1) that $x\vee p\not\leq q$, a contradiction. Therefore, $(x^2\vee p)_q\not=q$, and thus $n\geq 2$. Consequently, $p\leq (x^2\vee p)_q=(q^n)_q\leq (q^2)_q$. If $q=(q^2)_q$, then $q=(q^k)_q$ for each $k\in\mathbb{N}$, and hence $(x^2\vee p)_q=(q^n)_q=q$, a contradiction.
\end{proof}

\begin{lemma}\label{Lem 4.3} Let $\mathcal{L}$ be a $C$-lattice, and let $p<q$ be $\ell$-prime elements of $\mathcal{L}$. Suppose there is a compact weak meet principal element $x\leq q$ with $x_q=q$. If $p$ is a join of compact weak join principal elements, each of which is a product of $\ell$-radical elements, then $p=0_q$.
\end{lemma}

\begin{proof} Let $z$ be a compact weak join principal element with $z\leq p$ such that $z=z_1\cdots z_n$, where each $z_i$ is an $\ell$-radical element of $\mathcal{L}$. Since $p$ is a join of such elements, to prove that $p=0_q$ it suffices to show that $z_q\leq 0_q$. Now $z\leq p<q=x_q$, so without loss of generality, $z_1\leq p$. If $x\leq p$, then $q=x_q\leq p_q=p$, a contradiction. Consequently, $x\not\leq p$. Since $x$ is weak meet principal, there is $a\in\mathcal{L}$ such that $z_1\wedge x=ax$. Therefore, $(ax)_q=(z_1)_q\wedge x_q=(z_1)_q$.

Moreover, since $z_1\leq p$, it follows that $ax\leq (ax)_q=(z_1)_q\leq p$, and hence $a\leq p\leq q$. Therefore, we have that $a^2\leq a q=a x_q\leq (z_1)_q$. Since $(z_1)_q$ is an $\ell$-radical element of $\mathcal{L}$, this implies $a\leq (z_1)_q$. Hence $(z_1)_q=(ax)_q\leq (z_1x)_q$, and so $z_q=(z_1z_2\cdots z_n)_q\leq (z_1xz_2\cdots z_n)_q=(zx)_q$. Since $z\leq (zx)_q$ and $z$ is compact, Lemma~\ref{Lem 2.7} implies there is $b\not\leq q$ such that $zb\leq zx$. By assumption, $z$ is weak join principal, so $b\leq x\vee (0:z)$. Since $x\leq q$ and $b\not\leq q$, this implies $(0:z)\not\leq q$. Therefore, $z_q\leq 0_q$.
\end{proof}

\begin{lemma}\label{Lem 4.4} Let $\mathcal{L}$ be a principally generated radical factorial $C$-lattice.
\begin{enumerate}
\item $\dim\mathcal{L}\leq 1$.
\item If $x\in\mathcal{L}$ and $m\in\mbox{\rm Max}(\mathcal{L})$, then either $x_m=m^k$ for some $k\in\mathbb{N}_0$ or $x_m=0_m$.
\item $\mathcal{L}$ is a Pr\"ufer lattice, i.e., each compact element is $\ell$-principal.
\end{enumerate}
\end{lemma}

\begin{proof} (1) The main part of the proof of (1) consists in showing that if $p$ is a nonmaximal $\ell$-prime element, then $p=0_m$ for all $m\in\mbox{\rm Max}(\mathcal{L})$ above $p$. For suppose that this has been established and $q\leq p<m$ are $\ell$-prime elements and $m\in\mbox{\rm Max}(\mathcal{L})$. Then $p=0_m=q$ by the claim, which in turn implies that $\dim\mathcal{L}\leq 1$. Therefore, we focus in the proof on showing that if $p$ is a nonmaximal $\ell$-prime element, then $p=0_m$ for all $m\in\mbox{\rm Max}(\mathcal{L})$ above $p$.

Let $p$ be a nonmaximal $\ell$-prime element in $\mathcal{L}$, let $m\in\mbox{\rm Max}(\mathcal{L})$ with $p<m$, and let $x$ be an $\ell$-principal element in $\mathcal{L}$ such that $x\leq m$ and $x\not\leq p$. Let $q$ be an $\ell$-prime element with $q\leq m$ and $q$ minimal over $p\vee x$. By Lemma~\ref{Lem 4.2}, $q\ne (q^2)_q$. Let $y$ be an $\ell$-principal element in $\mathcal{L}$ such that $ y\leq q$ and $y\not\leq (q^2)_q$. Write $y=y_1\cdots y_k$, where the $y_i$ are $\ell$-radical elements. Then $y_q=((y_1)_q\cdots (y_k)_q )_q$. With an aim of applying Lemma~\ref{Lem 4.3}, we show that $y_q=q$.

In light of Lemma~\ref{Lem 2.8}, the preceding decomposition of $y$ and the assumption that $y\not\leq (q^2)_q$, to prove that $y_q=q$ it is enough to show that $q$ is a minimal $\ell$-prime element above $y$. Assume that there is some $\ell$-prime element $n$ with $y\leq n<q$. Then there exists an $\ell$-principal element $z\leq q$ and $z\not\leq n$.
Since $z\vee n\leq q$, there is an $\ell$-prime element $n'\leq q$ minimal over $z\vee n$. By assumption, $z^2\vee n$ is a product of $\ell$-radical elements. From Lemma~\ref{Lem 4.2} we have that $y\leq n\leq ((n')^2)_{n'}\not=n'$. Note that $((n')^2)_{n'}=v_1\cdots v_s$, where the $v_i$ are $\ell$-radical elements of $\mathcal{L}$. There is some $1\leq j\leq s$ such that $v_j\leq n'$. Since $(n')^2\leq v_j\leq n'$, we have that $v_j=n'$. If $v_i\not\leq n'$ for all $1\leq i\leq s$ such that $i\not=j$, then $((n')^2)_{n'}=((v_1)_{n'}\cdots (v_s)_{n'})_{n'}=n'$. We infer that $((n')^2)_{n'}=(n')^2$, and thus $y\leq (n')^2\leq (q^2)_q$, a contradiction.

Now, applying Lemma~\ref{Lem 4.3}, we have that $p=0_q$. If $q<m$, then we may repeat the preceding argument to show that there is an $\ell$-prime element $n$ with $q<n\leq m$ such that $q=0_n$. Then $p=0_q=(0_n)_q=q_q=q$, contrary to the choice of $q$. Therefore, $q=m$. We conclude that if $m\in\mbox{\rm Max}(\mathcal{L})$ with $p<m$, we have that $p=0_m$. This proves the claim.

(2) Let $m\in\mbox{\rm Max}(\mathcal{L})$, and let $x\in\mathcal{L}$. Without restriction let $x\leq m$. Let $p\in\mathcal{L}$ be $\ell$-prime with $p\leq m$ and $p$ minimal over $x_m$. Suppose first that $p<m$. Then as observed in the proof of (1), $p=0_m$, in which case, $x_m=0_m$. Otherwise, if $p=m$, then $\sqrt{x_m}=m$, and since $x_m$ is a product of $\ell$-radical elements, we have that $x_m$ is a power of $m$.

(3) Let $x\in\mathcal{L}$ be compact. It follows from Lemma~\ref{Lem 2.7}(6) that $x$ is $\ell$-principal if and only if $x_m$ is $\ell$-principal in $\mathcal{L}_m$ for each $m\in\mbox{\rm Max}(\mathcal{L})$ with $x\leq m$. Let $m\in\mbox{\rm Max}(\mathcal{L})$ with $x\leq m$. We show that $x_m$ is $\ell$-principal in $\mathcal{L}_m$. Indeed, by (2) the elements in $\mathcal{L}_m$ are totally ordered with respect to $\leq$. Since $x$ is compact in $\mathcal{L}$, $x_m$ is compact in $\mathcal{L}_m$. Since every element of $\mathcal{L}$, hence of $\mathcal{L}_m$ is a join of $\ell$-principal elements, it follows that $x_m$ is a join of finitely many $\ell$-principal elements in $\mathcal{L}_m$. Since the elements in $\mathcal{L}_m$ are totally ordered, it follows that $x_m$ is $\ell$-principal in $\mathcal{L}_m$.
\end{proof}

We want to thank T. Dumitrescu for pointing out (by personal communication) that Lemma~\ref{Lem 4.4}(1) can be proved directly (i.e., not relying on Lemmas~\ref{Lem 4.2} and~\ref{Lem 4.3} of this paper) by modifying the proof of \cite[Theorem 3.3]{AD}.

\begin{remark}\label{Rem 4.5} The proof of (3) shows that for each $m\in\mbox{\rm Max}(\mathcal{L})$, every element of $\mathcal{L}_m$ is an $\ell$-principal element. For more on $C$-lattices with this property (which are called almost principal element lattices), see \cite{JJA}.
\end{remark}

\begin{theorem}\label{Thm 4.6} The following are equivalent for a principally generated $C$-lattice domain $\mathcal{L}$.
\begin{enumerate}
\item $\mathcal{L}$ is a radical factorial lattice.
\item $\dim\mathcal{L}\leq 1$ and each $\ell$-invertible element is a product of $\ell$-radical elements.
\item Each nonzero $\ell$-prime element is maximal and above an $\ell$-invertible $\ell$-radical element.
\item Each element is a product of $\ell$-radical elements $x_1\leq\cdots\leq x_n$.
\item The $\ell$-radical of each nonzero compact element is $\ell$-invertible.
\item Every nonzero compact element is $\ell$-invertible and the $\ell$-radical of every compact element is compact.
\end{enumerate}
\end{theorem}

\begin{proof} (1) $\Rightarrow$ (2) This is an immediate consequence of Lemma~\ref{Lem 4.4}(1).

(2) $\Rightarrow$ (3) Let $m\in\mbox{\rm Max}(\mathcal{L})$, and let $x$ be an $\ell$-invertible element with $x\leq m$. By (2), $x=x_1\cdots x_k$ for some $\ell$-radical elements $x_1,\ldots,x_k$. Since $x$ is $\ell$-invertible, we have by Lemma~\ref{Lem 2.5} that so is each $x_i$. Since $m$ is $\ell$-prime, there is $i$ such that $x_i\leq m$, which verifies (3).

(3) $\Rightarrow$ (4) We use Theorem~\ref{Thm 3.1} to prove this implication. Let $0 \ne y\in\mathcal{L}$, and let $m\in\mbox{\rm Max}(\mathcal{L})$ with $y\leq m$. By (3), the $\ell$-prime elements minimal over $y$ are maximal elements. Thus $y$ is zero-dimensional, and so by Theorem~\ref{Thm 3.1}, $y$ is a product of $\ell$-radical elements $y_1\leq y_2\leq\cdots\leq y_k$.

(4) $\Rightarrow$ (1) This is clear.

(4) $\Rightarrow$ (6) Let $x$ be a nonzero compact element. By Lemma~\ref{Lem 4.4}(3), $x$ is $\ell$-principal. Since $\mathcal{L}$ is a multiplicative lattice domain, we have that $x$ is $\ell$-invertible. Now by (4), $x=x_1\cdots x_k$ for some $\ell$-radical elements $x_1\leq\cdots\leq x_k$. Since $x$ is $\ell$-invertible, so is $x_1$ by Lemma~\ref{Lem 2.5}. Since $x_1=\sqrt{x}$, statement (6) follows.

(6) $\Rightarrow$ (5) This is obvious.

(5) $\Rightarrow$ (3) We prove this part by establishing a series of claims.

\medskip

{\noindent}{\textsc{Claim 1}}: Every nonzero $\ell$-prime element of $\mathcal{L}$ is above an $\ell$-invertible $\ell$-radical element.

\medskip
Let $p\in\mathcal{L}$ be a nonzero $\ell$-prime element. There is some nonzero $x\in\mathcal{L}^*$ such that $x\leq p$. Set $y=\sqrt{x}$. Then $y$ is an $\ell$-invertible $\ell$-radical element of $\mathcal{L}$ and $y\leq p$. This proves Claim 1.

\medskip

{\noindent}{\textsc{Claim 2}}: If $x\in\mathcal{L}^*$ is nonzero and $p\in\mathcal{L}$ is an $\ell$-prime element minimal above $x$, then $p$ is a minimal nonzero $\ell$-prime element of $\mathcal{L}$.

\medskip

Let $x\in\mathcal{L}^*$ be nonzero, $p\in\mathcal{L}$ an $\ell$-prime element minimal above $x$ and $q\in\mathcal{L}$ an $\ell$-prime element such that $0<q\leq p$. We have to show that $p\leq q$. Observe that $\mathcal{L}_p$ is a principally generated $C$-lattice domain and $\mbox{\rm Max}(\mathcal{L}_p)=\{p_p\}$. Moreover, if $w\in\mathcal{L}_p$ is a nonzero compact element, then $w=t_p$ for some nonzero compact element $t\in\mathcal{L}$, and hence $\sqrt{w}=\sqrt{t_p}=\sqrt{t}_p$ is an $\ell$-invertible element of $\mathcal{L}_p$ (since $\sqrt{t}$ is an $\ell$-invertible element of $\mathcal{L}$). Therefore, the $\ell$-radical of every nonzero compact element of $\mathcal{L}_p$ is $\ell$-invertible. Also note that $x_p$ is a compact element of $\mathcal{L}_p$, $p_p$ is an $\ell$-prime element of $\mathcal{L}_p$ that is minimal above $x_p$, and $q_p$ is an $\ell$-prime element of $\mathcal{L}_p$ such that $0_p<q_p\leq p_p$. Therefore, we can assume without restriction that $\mbox{\rm Max}(\mathcal{L})=\{p\}$.

Since $p$ is minimal above $x$, we have that $p=\sqrt{x}$ is an $\ell$-invertible element of $\mathcal{L}$. By Claim 1 there is an $\ell$-invertible $\ell$-radical element $y\in\mathcal{L}$ such that $y\leq q$. We have that $y\leq p$, and since $p$ is weak meet principal, there is some $b\in\mathcal{L}$ such that $y=pb$. Assume that $b\not=1$. Since $b\leq p$ and $p$ is weak meet principal, there is some $d\in\mathcal{L}$ such that $b=pd$. Consequently, $(pd)^2\leq p^2d=y$, and since $y$ is $\ell$-radical, we have that $pd\leq y$. Therefore, $b=pd\leq y\leq b$. It follows that $y=b$, and thus $y=py$. Since $y$ is cancellative, we infer that $p=1$, a contradiction. This implies that $b=1$, and hence $p=y\leq q$, which proves Claim 2.

\medskip

{\noindent}{\textsc{Claim 3}}: For each $m\in\mbox{\rm Max}(\mathcal{L})$ and all $\ell$-invertible $x,y\in\mathcal{L}$ we have that $x_m$ and $y_m$ are comparable.

\medskip

Assume to the contrary that there are some $m\in\mbox{\rm Max}(\mathcal{L})$ and $\ell$-invertible elements $x,y\in\mathcal{L}$ such that $x_m$ and $y_m$ are not comparable. Observe that $x,y\leq m$. Clearly, $x\vee y$ is nonzero and compact, and thus $z_0=\sqrt{x\vee y}$ is $\ell$-invertible and $\ell$-radical. Since $x\leq z_0$, $y\leq z_0$ and $z_0$ is weak meet principal, there are some $v,w\in\mathcal{L}$ such that $x=vz_0$ and $y=wz_0$. We have that $v$ and $w$ are $\ell$-invertible. If $v_m$ and $w_m$ are comparable, say $v_m\leq w_m$, then $x_m=(v_m(z_0)_m)_m\leq (w_m(z_0)_m)_m=y_m$, a contradiction. Therefore, $v_m$ and $w_m$ are not comparable.

Set $x_0=x$ and $y_0=y$. Using the observation before, we can recursively construct a sequence $(z_i)_{i\in\mathbb{N}_0}$ of $\ell$-invertible $\ell$-radical elements of $\mathcal{L}$ and sequences $(x_i)_{i\in\mathbb{N}_0}$ and $(y_i)_{i\in\mathbb{N}_0}$ of $\ell$-invertible elements of $\mathcal{L}$ such that $z_i=\sqrt{x_i\vee y_i}$, $x_i=x_{i+1}z_i$ and $y_i=y_{i+1}z_i$ for all $i\in\mathbb{N}_0$. Note that if $i\in\mathbb{N}_0$, then $(x_i)_m$ and $(y_i)_m$ are not comparable and, in particular, $x_i,y_i\leq m$ and $z_i\leq z_{i+1}\leq m$. Observe that $z_0$ is compact, and hence there is some $k\in\mathbb{N}$ such that $z_0^k\leq x_0\vee y_0$. We infer that $z_0^k\leq x_0\vee y_0=z_0\cdots z_{k-1}(x_k\vee y_k)\leq z_0\cdots z_k$. Since $z_k\not=1$, there is some $\ell$-prime element $p\in\mathcal{L}$ that is minimal above $z_k$. Since $z_k$ is nonzero and compact, it follows by Claim 2 that $p$ is a minimal nonzero $\ell$-prime element of $\mathcal{L}$. In particular, if $0\leq i\leq k$, then $p$ is minimal above $z_i$ and $(z_i)_p=p$. Consequently, $(p^k)_p=(z_0^k)_p\leq (z_0\cdots z_k)_p=(p^{k+1})_p$, and thus $(p^k)_p=(p^{k+1})_p$. Note that $(z_0)_p=p$ and $(z_0^k)_p=(z_0^{k+1})_p$. Observe that $(z_0)_p$ is an $\ell$-invertible element of $\mathcal{L}_p$. Therefore, $(z_0)_p=1$ and $z_0\leq z_k\leq p$, a contradiction. This proves Claim 3.

\medskip

Now let $p$ be an $\ell$-prime element and $m\in\mbox{\rm Max}(\mathcal{L})$ such that $0<p\leq m$. It is sufficient to show that $b\leq p$ for all nonzero $b\in\mathcal{L}^*$ such that $b\leq m$. Let $b\in\mathcal{L}^*$ be nonzero such that $b\leq m$. There is some $\ell$-prime element $q\in\mathcal{L}$ that is minimal above $b$ such that $q\leq m$. Observe that $p=\bigvee\{c_m\mid c\in\mathcal{L}\textnormal{ is }\ell\textnormal{-invertible and }c\leq p\}$ and $q=\bigvee\{c_m\mid c\in\mathcal{L}\textnormal{ is }\ell\textnormal{-invertible and }c\leq q\}$. Assume that $p$ and $q$ are not comparable. Then there are some $\ell$-invertible elements $c,d\in\mathcal{L}$ such that $c_m\leq p$, $c_m\not\leq q$, $d_m\not\leq p$ and $d_m\leq q$. We infer that $c_m$ and $d_m$ are not comparable, which contradicts Claim 3. Therefore, $p$ and $q$ are comparable. If $p\leq q$, then since $q$ is a minimal nonzero $\ell$-prime element by Claim 2, it follows that $p=q$. In any case we have that $b\leq q\leq p$.
\end{proof}

We will prove in Corollary~\ref{Cor 6.7} that the decomposition $x=x_1\cdots x_n$ in statement (4) of Theorem~\ref{Thm 4.6} is unique if $x$ is nonzero and the $x_i$ are proper. The following remark was communicated to us by T. Dumitrescu.

\begin{remark}\label{Rem 4.7} Let $\mathcal{L}$ be a principally generated radical factorial $C$-lattice domain. Then $\mathcal{L}$ is isomorphic to the lattice of ideals of some SP-domain.
\end{remark}

\begin{proof} By Lemma~\ref{Lem 4.4}(2), $\mathcal{L}$ is locally totally ordered. This implies that $\mathcal{L}_m$ is a modular lattice for each $m\in\mbox{\rm Max}(\mathcal{L})$, and hence $\mathcal{L}$ is a modular lattice. Now \cite[Theorem 3.4]{And} applies to show that $\mathcal{L}$ is isomorphic to the lattice of ideals of some Pr\"ufer domain $D$. Note that the $\ell$-radical elements of the lattice of ideals of $D$ are precisely the radical ideals of $D$. Since $\mathcal{L}$ is a radical factorial lattice, we infer that $D$ is an SP-domain.
\end{proof}

\begin{remark}\label{Rem 4.8} Note that the ``$\mathcal{L}$ is a principally generated lattice'' condition in Lemma~\ref{Lem 4.4} and Theorem~\ref{Thm 4.6} cannot be replaced by the condition that every element of $\mathcal{L}$ is a join of (weak) meet principal cancellative elements. We consider the monoid $H$ that is constructed in \cite[Example 4.2]{RR}, and we let $\mathcal{L}$ be the set of $t$-ideals of $H$. (The definition of $t$-ideals and the $t$-system of monoids can be found in \cite{RR}.)

Note that $H$ is an (additively written) cancellative monoid and $t$ is a finitary ideal system on $H$. By Lemma~\ref{Lem 8.1} we know that $\mathcal{L}$ is a $C$-lattice. It follows by \cite[Example 4.2]{RR} that $\mathcal{L}$ is a radical factorial lattice such that $\dim\mathcal{L}=2$ and every nonzero element of $\mathcal{L}$ is cancellative. In particular, $\mathcal{L}$ is a multiplicative lattice domain. Moreover, if $I\in\mathcal{L}$, then $I=\bigvee\{x+H\mid x\in I\}$ and $x+H$ is a meet principal cancellative element of $\mathcal{L}$ for each $x\in H$.
\end{remark}

\section{Example: Upper semicontinuous functions}

The purpose of this section is to give a class of examples of radical factorial lattices arising in a topological context. The importance of this class becomes evident in the next section, where it is shown that all principally generated radical factorial $C$-lattice domains arise this way. In later sections, we interpret these topological results in the context of rings and monoids.

Recall that for a topological space $X$, a function $f:X\rightarrow\mathbb{N}_0={\mathbb{N}}\cup\{0\}$ is {\it upper semicontinuous} if $f^{-1}([n,\infty))$ is a closed set for all $n\in\mathbb{N}_0$. Our focus is on compactly supported upper semicontinuous functions taking values among the nonnegative integers ${\mathbb{N}}_0$. If $X$ is a Hausdorff space, then compact subsets of $X$ are closed and so we have that a function $f:X\rightarrow\mathbb{N}_0$ is compactly supported and upper semicontinuous if and only if $f^{-1}([k,\infty))$ is compact for all $k\in\mathbb{N}$. The compactness of the preimages here implies that such a function takes on only finitely many values. A convenient decomposition of such functions is given in Lemma~\ref{Lem 5.2}(5) based on this observation.

\begin{definition}\label{Def 5.1} Let $X$ be a Hausdorff space. We define $U(X)$ to be the monoid of compactly supported upper semicontinuous functions $f:X\rightarrow\mathbb{N}_0$ with binary operation given by pointwise addition of functions. We define an order $\leq_d$ on $U(X)$ dual to the usual one by $f\leq_d g$ iff $f(x)\geq g(x)$ for all $x\in X$. With this order, the zero function $0$ is the top element of $U(X)$ and is an additive identity for the monoid $(U(X),+)$. For technical reasons, it will be convenient to introduce a bottom element $b$ to this partially ordered set. We define $b:X\rightarrow\mathbb{N}_0\cup\{\infty\}$ by $b(x)=\infty$ for each $x\in X$. Thus $U_b(X):=U(X)\cup\{b\}$ is the partially ordered set with this bottom element appended. Observe that $b+f=f+b=b$ for all $f\in U_b(X)$, and hence $U_b(X)$ is a monoid.
\end{definition}

Throughout this section, we denote the characteristic function of a subset $A$ of a set $X$ by $\mathbf{1}_A$; i.e., $\mathbf{1}_A(x)=1$ if $x\in A$ and $\mathbf{1}_A(x)=0$ if $x\not\in A$.

\begin{lemma}\label{Lem 5.2} Let $X$ be a Hausdorff space. With the order $\leq_d$, $U_b(X)$ is a multiplicative lattice domain with the following properties.
\begin{enumerate}
\item The meet $f\wedge_d g$ and join $f\vee_d g$ are given for all $x\in X$ by
\begin{center} $(f\wedge_d g)(x)=\max\{f(x),g(x)\}$ and $(f\vee_d g)(x)=\min\{f(x),g(x)\}$.\end{center}
\item If $\mathcal{F}$ is a nonempty subset of $U(X)$, then $\bigvee_d\mathcal{F}$ exists and is given for all $x\in X$ by \begin{center}$(\bigvee_d\mathcal{F})(x)=\min\{f(x)\mid f\in\mathcal{F}\}.$\end{center}
\item If $\mathcal{F}$ is a subset of $U(X)$ that is bounded below in $U(X)$ with respect to $\leq_d$, then $\bigwedge_d\mathcal{F}$ exists and is given by $\bigwedge_d\mathcal{F}=\bigvee_{d}{\mathcal{G}}$, where \begin{center} ${\mathcal{G}}=\{g\in U(X)\mid g(x)\geq f(x)$ {\rm{for all}} $x\in X $ {\rm{and}} $f\in\mathcal{F}\}.$\end{center}
\item If $\mathcal{F}$ is a subset of $U(X)$ that is not bounded below in $U(X)$, then $\bigwedge_d\mathcal{F} = b$ in $U_b(X)$.
\item A function $f:X\rightarrow {\mathbb{N}}_0$ is in $U(X)$ iff $$f=k_0\mathbf{1}_{C_0}+\sum_{i=1}^n (k_{i}-k_{i-1})\mathbf{1}_{C_i},$$ where $k_0<k_1<\cdots<k_n$ are positive integers and $C_0\supseteq C_1\supseteq\cdots\supseteq C_n$ are compact subsets of $X$.
\end{enumerate}
\end{lemma}

\begin{proof} That $U_b(X)$ is a multiplicative lattice (written additively) follows from (1)--(4), which we will establish below, and the observation that addition (since it is pointwise) commutes with the arbitrary join defined in (2). That the bottom element $b$ of $U_b(X)$ is $\ell$-prime is a consequence of the fact that $f+g\in U(X)$ for all $f,g\in U(X)$. Thus, once we have established (1)--(4), we have that  $U_b(X)$ is a multiplicative lattice domain for which the top element of $U_b(X)$ is the zero function.

(1) Define a function $h$ on $X$ by $h(x)=\max\{f(x),g(x)\}$ for all $x\in X$. For each $k\in\mathbb{N}$, we have that $h^{-1}([k,\infty))=f^{-1}([k,\infty))\cup g^{-1}([k,\infty))$. As a union of two compact sets, $h^{-1}([k,\infty))$ is also compact. Therefore, $h\in U(X)$. It is clear that $h$ is the greatest lower bound of $f$ and $g$ with respect to $\leq_d$.
The proof that the join exists and is as claimed is similar, using instead the fact that the intersection of compact sets is closed, hence compact.

(2) The proof of (2) is a straightforward extension of the argument in (1).

(3) This follows from (2).

(4) This is clear.

(5) Let $f\in U(X)$. Since $f$ is compactly supported,
$f$ is bounded, and so $f$ takes on only finitely many positive values, say $k_0<k_1<\cdots<k_n$ are the positive values of $f$. For each $i$, let $C_i=f^{-1}([k_i,\infty))$. Since $f$ is compactly supported and each $k_i$ is postive, each closed set $C_i$ is compact. Hence each characteristic function $\mathbf{1}_{C_i}$ is in $U(X)$ and ${C_0}\supseteq\cdots\supseteq {C_n}$. Moreover,
$$f=k_0\mathbf{1}_{C_0}+\sum_{i=1}^n (k_{i}-k_{i-1})\mathbf{1}_{C_i}.$$
Conversely, any function of this form is easily seen to be upper semicontinuous and compactly supported.
\end{proof}

\begin{theorem}\label{Thm 5.3} If $X$ is a Hausdorff space, then $U_b(X)$ is a radical factorial lattice domain. \end{theorem}

\begin{proof} By Lemma~\ref{Lem 5.2}, $U_b(X)$ is a multiplicative lattice domain. We claim that $\mathbf{1}_A$ is an $\ell$-radical element of $U_b(X)$ for each compact subset $A$ of $X$. Let $A$ be a compact subset of $X$. Then $\mathbf{1}_A\in U_b(X)$. To see that $\mathbf{1}_A$ is $\ell$-radical, suppose $f\in U_b(X)$ and $nf\leq_d\mathbf{1}_A$ for some $n\in\mathbb{N}$. Then $1=\mathbf{1}_A(x)\leq nf(x)$ for all $x\in A$. Therefore, for each $x\in A$, $f(x)\ne 0$ and hence $1=\mathbf{1}_A(x)\leq f(x)$. It follows that $\mathbf{1}_A(x)\leq f(x)$ for all $x\in X$ and hence $f\leq_d\mathbf{1}_A$. Thus $\mathbf{1}_A$ is an $\ell$-radical element of $U_b(X)$. Applying Lemma~\ref{Lem 5.2}(5), we obtain that $U_b(X)$ is a radical factorial lattice.
\end{proof}

\begin{remark}\label{Rem 5.4} The proof of Theorem~\ref{Thm 5.3} shows that for each compact subset $A$ of $X$, ${\bf 1}_A$ is an $\ell$-radical element of $U(X)$. The converse is also true: Suppose $g$ is an $\ell$-radical element of $U(X)$. If $g=0$, then $g$ is the characteristic function of the empty set. Suppose $g\ne 0$. Since $g$ is bounded, $n=\max\{g(x)\mid x\in X\}$ exists. Since the values of $g$ are nonnegative integers, to prove $g$ is a characteristic function of a closed set, it suffices to show that $n=1$. Let $A=g^{-1}([1,\infty))$. Since $g$ is upper semicontinuous, $A$ is closed in $X$. Now $g(x)\leq n\mathbf{1}_A(x)$ for all $x\in X$. Thus $n\mathbf{1}_A\leq_d g$, and since $g$ is $\ell$-radical, we have that $\mathbf{1}_A\leq_d g$. But then $g(x)\leq 1$ for all $x\in A$. Therefore, $n=1$. Moreover, since $g$ is compactly supported, $A$ is a closed subset of a compact set and hence is compact.
\end{remark}

\section{Representation of radical factorial lattices}

In this section we show that every principally generated radical factorial $C$-lattice domain can be represented as the multiplicative lattice $U_b(X)$ of compactly supported upper semicontinuous functions studied in the last section. Using this fact, we show in  Corollary~\ref{Cor 6.6}
 that the structure of a principally generated radical factorial $C$-lattice domain is determined entirely by the topology of the space of maximal elements of $\mathcal{L}$. Using this description, we obtain in Corollary~\ref{Cor 6.7} a uniqueness result for the representation of elements as products of $\ell$-radical elements in such radical factorial lattices.

\begin{definition}\label{Def 6.1} For a $C$-lattice $\mathcal{L}$, we define $X_\mathcal{L}=\mbox{\rm Max}(\mathcal{L})$, and we view $X_\mathcal{L}$ as a topological space with respect to the {\it inverse topology} on $X_\mathcal{L}$. This topology has as a basis of open sets of the form $V(x):=\{m\in X_\mathcal{L}\mid x\leq m\}$, where $x$ is a compact element of $\mathcal{L}$. Thus $X_\mathcal{L}$ has a basis of closed sets of the form $U(x):=\{m\in X_\mathcal{L}\mid x\not\leq m\}$, where $x$ is compact.
\end{definition}

For the next lemma, recall that a topological space is {\it zero-dimensional} if it has a basis of clopen sets.

\begin{lemma}\label{Pro 6.2} If $\mathcal{L}$ is a principally generated radical factorial $C$-lattice domain, then $X_{\mathcal{L}}$ is a zero-dimensional Hausdorff space.
\end{lemma}

\begin{proof} To see that $X_\mathcal{L}$ is Hausdorff, let $m,n\in X_\mathcal{L}$. Then $m\vee n=1$, so, since $1$ is compact and each of $x$ and $y$ is a join of compact elements, there exist compact elements $x\leq m$ and $y\leq n$ such that $x\vee y=1$. Therefore, $V(x)$ and $V(y)$ are disjoint open neighborhoods of $m$ and $n$, respectively, proving that $X_\mathcal{L}$ is Hausdorff.

To prove that $X_{\mathcal{L}}$ is a zero-dimensional space, it suffices to show that for each nonzero compact element $y$ in $\mathcal{L}$, the set $V(y)$ is a compact open subspace of $X_\mathcal{L}$.
Let $y$ be a nonzero compact element in $\mathcal{L}$. By Theorem~\ref{Thm 4.6}, $\sqrt{y}$ is again compact. Thus we may assume without loss of generality that $y$ is $\ell$-radical.
Suppose $\{y_\alpha\}$ is a collection of compact elements in $\mathcal{L}$ such that $V(y)\subseteq\bigcup_\alpha V(y_\alpha)$. Then $$V(y)=\bigcup_\alpha\left(V(y_\alpha)\cap V(y)\right)=\bigcup_\alpha V(y\vee y_\alpha).$$

We claim that for each $\alpha$, $(y\vee y_\alpha)\vee (y:(y\vee y_\alpha))=1$. Let $m\in X_{\mathcal{L}}$. If $y\vee y_\alpha\leq m$, then since $\dim\mathcal{L} \leq 1$ by Theorem~\ref{Thm 4.6} and $y$ is $\ell$-radical and nonzero, we have that $y_m=(y\vee y_\alpha)_m=m$. Since $y\vee y_\alpha$ is compact, Lemma~\ref{Lem 2.7}(6) implies that $(y:(y\vee y_\alpha))_m=(y_m:(y\vee y_\alpha)_m)=(m:m)=1$. Therefore, if $y\vee y_\alpha\leq m$, we have that $(y:(y\vee y_\alpha))\not\leq m$. Consequently, $(y\vee y_\alpha)\vee (y:(y\vee y_\alpha))=1$.

Since $1$ is compact, we may choose a compact element $z_\alpha\in\mathcal{L}$ such that $z_\alpha\leq (y:(y\vee y_\alpha))$ and $(y\vee y_\alpha)\vee z_\alpha=1$.

Thus $$(y\vee y_\alpha)z_\alpha\leq (y\vee y_\alpha)(y:(y\vee y_\alpha))\leq y,$$ and so we conclude that $V(y\vee y_\alpha)=V(y)\cap U(z_\alpha).$ Therefore, $$V(y)=\bigcup_\alpha \left(V(y)\cap U(z_\alpha)\right).$$ It follows that $V(y) \subseteq \bigcup_\alpha U(z_\alpha)$ and hence
$y\vee (\bigvee_\alpha z_\alpha)=1$. Since $1$ is compact, there exist $\alpha_1,\ldots,\alpha_k$ such that $y\vee z_{\alpha_1}\vee\cdots\vee z_{\alpha_k}=1.$ Therefore,
$$ V(y)=\bigcup_{i=1}^k\left(V(y)\cap U(z_{\alpha_i})\right)=\bigcup_{i=1}^k V(y\vee y_{\alpha_i}).$$ This proves that the open cover $\{V(y)\cap V(y_\alpha)\}$ of $V(y)$ has a finite subcover. We conclude that $V(y)$ is compact.
\end{proof}

The next lemma introduces a valuation-like map ${\rm v}_m:\mathcal{L}\rightarrow {\mathbb{N}_0}\cup\{\infty\}$ for each $m\in X_\mathcal{L}$. This map is used in Lemma~\ref{Lem 6.4} to define the functions from $X_\mathcal{L}$ to ${\mathbb{N}}_0$ that will be our primary interest in this section.

\begin{lemma}\label{Lem 6.3} Let $\mathcal{L}$ be a principally generated radical factorial $C$-lattice domain, and let $m\in X_\mathcal{L}$. For each $m\in X_\mathcal{L}$, define ${\rm v}_m:\mathcal{L}\rightarrow {\mathbb{N}_0}\cup\{\infty\}$ by ${\rm v}_m(x)=\sup\{k\in\mathbb{N}_0\mid x\leq m^k\}$ for each $x\in\mathcal{L}$. The following properties hold for all nonzero $x,y\in\mathcal{L}$.
\begin{itemize}
\item[(i)] ${\rm v}_m(x)$ is the unique $k\in\mathbb{N}_0$ for which $x_m=m^k$.
\item[(ii)] ${\rm v}_m(xy)={\rm v}_m(x)+{\rm v}_m(y)$.
\end{itemize}
\end{lemma}

\begin{proof} (i) First we show that there is some $k\in\mathbb{N}_0$ for which $x_m=m^k$. Assume to the contrary that $x_m\not=m^k$ for all $k\in\mathbb{N}_0$. Then by Lemma~\ref{Lem 4.4}(2), we have that $x_m=0_m$. However, $\mathcal{L}$ is a principally generated $C$-lattice domain, and so $x_m$ is above an $\ell$-invertible element $y$ in $\mathcal{L}$. Thus $y_m=0_m=(y^2)_m$, so that $yb\leq y^2$ for some compact $b\not\leq m$. Since $y$ is $\ell$-invertible, this implies $b\leq y\leq m$, a contradiction.

It remains to show that for each $n\in\mathbb{N}_0$ with $x_m=m^n$ we have that $n={\rm v}_m(x)$. Let $n\in\mathbb{N}_0$ be such that $x_m=m^n$. Then $x\leq m^n$, and hence $n\leq {\rm v}_m(x)$. Suppose that $n<{\rm v}_m(x)$. Then $x_m\leq m^{n+1}\leq m^n=x_m$, and hence $m^n=m^{n+1}$. By Theorem~\ref{Thm 4.6}, there is a zero-dimensional $\ell$-invertible $\ell$-radical element $z$ with $z_m=m$. Thus $m^n=m^{n+1}$ implies $(z^n)_m=(z^{n+1})_m$. In this case, there is a compact $b\not\leq m$ such that $bz^n\leq z^{n+1}$. Since $z$ is $\ell$-invertible, this implies $b\leq z\leq m$, a contradiction.

(ii) By (i) and Lemma~\ref{Lem 2.7}(2) it follows that
\[
m^{{\rm v}_m(xy)}=(x_my_m)_m=(m^{{\rm v}_m(x)}m^{{\rm v}_m(y)})_m=m^{{\rm v}_m(x)+{\rm v}_m(y)}.
\]
We infer by (i) that ${\rm v}_m(xy)={\rm v}_m(x)+{\rm v}_m(y)$.
\end{proof}

\begin{lemma}\label{Lem 6.4} If $\mathcal{L}$ is a principally generated radical factorial $C$-lattice domain, then for each nonzero $x\in\mathcal{L}$, the function $\alpha_x:X_\mathcal{L}\rightarrow {\mathbb{N}_0}:m\mapsto {\rm v}_m(x)$ is compactly supported and upper semicontinuous. If $x$ is compact, then $\alpha_x$ is continuous.
\end{lemma}

\begin{proof} We prove the second assertion first. Let $x$ be a nonzero compact element of $\mathcal{L}$. Since ${\mathbb{N}_0}$ is a discrete space, to show that $\alpha_x$ is continuous it suffices to prove that for each $k\in {\mathbb{N}_0}$, $\{m\in X_\mathcal{L}\mid {\rm v}_m(x)=k\}$ is an open subset of $X_\mathcal{L}$. Let $k\in {\mathbb{N}_0}$. Using Theorem~\ref{Thm 4.6}, write $x=x_1\cdots x_s$ for some $\ell$-radical elements $x_1\leq\cdots\leq x_s$.

We observe that the join of any two nonzero $\ell$-radical elements $x,y\in\mathcal{L}$ is $\ell$-radical. Note that if $n\in X_\mathcal{L}$ is such that $x\vee y\leq n$, then $n\geq (x\vee y)_n\geq x_n\vee y_n=n\vee n=n$, and hence $\{(x\vee y)_m\mid m\in X_\mathcal{L},x\vee y\leq m\}=\{m\in X_\mathcal{L}\mid x\vee y\leq m\}$. Since $\dim(\mathcal{L})\leq 1$ by Theorem~\ref{Thm 4.6}, we infer from Lemma~\ref{Lem 2.7} that
\begin{eqnarray*}
\sqrt{x\vee y}&=&\bigwedge\{m\in X_\mathcal{L}\mid x\vee y\leq m\}\\
\:&=&\bigwedge\{(x\vee y)_m\mid m\in X_\mathcal{L} \ {\rm and } \ x\vee y\leq m\} \:=\:x\vee y.
\end{eqnarray*}

Moreover, for any two $\ell$-radical elements $y$ and $z$ in $\mathcal{L}$ with $y\not\leq z$, we have that $$(z:y)=\bigwedge\{m\in X_\mathcal{L}\mid m\geq z, m\not\geq y\}=\bigwedge\left(V(z)\cap U(y)\right).$$
Set $x_{s+1}=1$. Using these observations and the assumption that the $x_i$'s form a chain, we see that
\begin{eqnarray*}
\{m\in X_\mathcal{L}\mid {\rm v}_m(x)=k\}&=&\{m\in X_\mathcal{L}\mid {\rm v}_m(x_1\cdots x_s)=k\}\\
\:&=&\{m\in X_\mathcal{L}\mid x_k\leq m,x_{k+1}\not\leq m\}\\
\:&=&\{m\in X_\mathcal{L}\mid (x_k:x_{k+1})\leq m\}\:=\:V((x_k:x_{k+1})).
\end{eqnarray*}

To prove this set is open in $X_\mathcal{L}$, it suffices to show that $(x_k:x_{k+1})$ is compact. Since $x$ is compact, we have by Lemma~\ref{Lem 4.4}(3) that $x$ is $\ell$-invertible, and hence by Lemma~\ref{Lem 2.5}(2) each $x_i$ is $\ell$-invertible. From this it follows that $(x_k:x_{k+1})$ is compact: If $(x_k:x_{k+1})\leq\bigvee_\alpha y_\alpha$, then $x_k=x_{k+1}(x_k:x_{k+1})\leq\bigvee_\alpha x_{k+1}y_\alpha$. The compactness of $x_k$ implies then that $x_k\leq x_{k+1}y_{\alpha_1}\vee\cdots\vee x_{k+1}y_{\alpha_t}$ for some $\alpha_1,\ldots,\alpha_t$. Using the fact that $x_{k+1}$ is cancellative, we obtain $$(x_k:x_{k+1})\leq (x_{k+1}y_{\alpha_1}\vee\cdots\vee x_{k+1}y_{\alpha_t}:x_{k+1})=y_{\alpha_1}\vee\cdots\vee y_{\alpha_t},$$ which proves that $(x_k:x_{k+1})$ is compact. This shows that $\alpha_x$ is continuous.

Next, suppose that $x$ is an element of $X_\mathcal{L}$ that is not necessarily compact. We claim that $\alpha_x$ is upper semicontinuous. To this end, let $k\in\mathbb{N}_0$. Let $A$ be the set of $\ell$-invertible elements in $\mathcal{L}$ below $x$. Since $\mathcal{L}$ is a principally generated $C$-lattice domain, we have that $x=\bigvee_{a\in A}a$. Therefore,
\begin{eqnarray*}
\alpha_x^{-1}([k,\infty))&=&\{m\in X_\mathcal{L}\mid {\rm v}_m(x)\geq k\}\:
\:=\:\:\{m\in X_\mathcal{L}\mid x\leq m^k\}\\
&=&\{m\in X_\mathcal{L}\mid a\leq m^k {\mbox{ for all }}a\in A\}\\
&=&\{m\in X_\mathcal{L}\mid {\rm v}_m(a)\geq k {\mbox{ for all }}a\in A\}\\
&=&\bigcap_{a\in A}\alpha_{a}^{-1}([k,\infty)).
\end{eqnarray*}

Since each $a$ is compact, $\alpha_{ a}$ is a continuous function by what we have previously established. Therefore, the last intersection is an intersection of closed sets. Hence $\alpha_x^{-1}([k,\infty))$ is closed, which proves that $\alpha_x$ is upper semi-continuous.

To see next that $\alpha_x$ is compactly supported, let $y$ be an $\ell$-invertible element in $\mathcal{L}$ with $y\leq x$. To prove that $x$ is compactly supported, it suffices to show that $\{m\in X_\mathcal{L}\mid y\leq m\}$ is compact in $X_\mathcal{L}$. This is the case by Lemma~\ref{Pro 6.2}, so $\alpha_x$ is compactly supported.
\end{proof}

\begin{theorem}\label{Thm 6.5} If $\mathcal{L}$ is a principally generated radical factorial $C$-lattice domain, then $\mathcal{L}$ and $U_b(X_\mathcal{L})$ are isomorphic as multiplicative lattices.
\end{theorem}

\begin{proof}
We claim that the mapping $\phi:\mathcal{L}\rightarrow U_b(X_\mathcal{L})$ defined by $\phi(0) = b$ and $\phi(x) = \alpha_x$ for all nonzero $x \in \mathcal{L}$   is an isomorphism of multiplicative lattices. By Lemma~\ref{Lem 6.4}, $\phi$ is well-defined. To see that $\phi$ is a homomorphism of monoids, first observe that if $x \in \mathcal{L}$, then it is clear that $\phi(0\cdot x) = \phi(0) + \phi(x)$.  Now let  $x,y$ be nonzero elements in $\mathcal{L}$. Then for each $m\in X_\mathcal{L}$, $\phi(xy)(m)=\alpha_{xy}(m)={\rm v}_m(xy)={\rm v}_m(x)+{\rm v}_m(y)=\alpha_{x}(m)+\alpha_{y}(m)=(\phi(x)+\phi(y))(m)$ by Lemma~\ref{Lem 6.3}(ii) and $\phi(1)(m)=\alpha_1(m)={\rm v}_m(1)=0$. Therefore, $\phi(1)=\mathbf{0}$ and $\phi(xy)=\phi(x)+\phi(y)$ for all  $x,y \in \mathcal{L}$, and thus $\phi$ is a homomorphism of monoids.

We show next that for all $x,y\in\mathcal{L}$, $x\leq y$ iff $\phi(x)\leq_d\phi(y)$. Let $x,y\in\mathcal{L}$ and $m\in X_\mathcal{L}$. Without restriction let $x,y\not=0$. If $x\leq y$, then $\phi(x)(m)={\rm v}_m(x)\geq {\rm v}_m(y)=\phi(y)(m)$, so that $\phi(x)(n)\geq\phi(y)(n)$ for all $n\in X_\mathcal{L}$, and hence $\phi(x)\leq_d\phi(y)$. Conversely, if $\phi(x)\leq_d\phi(y)$, then ${\rm v}_m(y)=\phi(y)(m)\leq\phi(x)(m)={\rm v}_m(x)$, so that $x_n=n^{{\rm v}_n(x)}\leq n^{{\rm v}_n(y)}=y_n$ for all $n\in X_\mathcal{L}$ by Lemma~\ref{Lem 6.3}(i), and hence $x\leq y$ by Lemma~\ref{Lem 2.7}(5).

It is an immediate consequence of the last statement that $\phi$ is injective.

Finally, to see that $\phi$ is onto, observe first that $\phi(0) = b$ and  every element of $U(X_\mathcal{L})$ is by Lemma~\ref{Lem 5.2}(5) a linear combination of characteristic functions $\mathbf{1}_{C}$, where $C$ is compact in $X_\mathcal{L}$. Therefore, we need only show that each such characteristic function is in the image of $\phi$. Let $C$ be a compact subset of $X_\mathcal{L}$. Without restriction we can assume that $0\not\in X_{\mathcal{L}}$. Since $C\subseteq\bigcup_{a\in\mathcal{L}^*\setminus\{0\}} V(a)$, the fact that $C$ is compact implies there is a finite set $A\subseteq\mathcal{L}^*\setminus\{0\}$ such that $C\subseteq\bigcup_{a\in A} V(a)$. Set $y=\sqrt{\prod_{a\in A} a}$. Then $y$ is a nonzero $\ell$-radical element of $\mathcal{L}$ such that $C\subseteq V(y)$.

Since $C$ is closed in $X_\mathcal{L}$, there is a collection $\{y_i\}$ of compact elements in $\mathcal{L}$ such that $C=\bigcap_{i} U(y_i)$. We can assume without restriction that each $y_i$ is $\ell$-radical, since $U(y_i)=U(\sqrt{y_i})$ and $\sqrt{y_i}$ is compact for each $i$ by Theorem 4.6. Since also $C\subseteq V(y)$, we have that $C=\bigcap_{i}\left(V(y)\cap U(y_i)\right)$. Now $V(y)\cap U(y_i)=V((y:y_i))$. Let $z=\bigvee_i (y:y_i)$. Then $C=V(z)$ and $z$ is an $\ell$-radical element in $\mathcal{L}$ since every element above the zero-dimensional $\ell$-radical element $y$ is $\ell$-radical. Thus $z_m=m$ for each $m\in\mbox{\rm Max}(\mathcal{L})$ above $z$. Since $\phi(z)(m)=\alpha_z(m)={\rm v}_m(z)$, it follows that for each $m\in X_\mathcal{L}$, we have that $m\in C$ if and only if $\phi(z)(m)=1$. Therefore, $\phi(z)=\mathbf{1}_C$, which proves that $\phi$ is onto.
\end{proof}

\begin{corollary}\label{Cor 6.6} Two principally generated radical factorial $C$-lattice domains $\mathcal{L}$ and $\mathcal{L}'$ are isomorphic if and only if $X_\mathcal{L}$ and $X_{\mathcal{L}'}$ are homeomorphic.
\end{corollary}

\begin{proof} If $X_\mathcal{L}$ and $X_{\mathcal{L}'}$ are homeomorphic, then it follows that $U_b(X_\mathcal{L})\cong U_b(X_{\mathcal{L}'})$. Theorem~\ref{Thm 6.5} then implies $\mathcal{L}\cong\mathcal{L}'$. The converse is straightforward.
\end{proof}

\begin{corollary}\label{Cor 6.7} Let $\mathcal{L}$ be a principally generated radical factorial $C$-lattice domain. Then each nonzero $x\in\mathcal{L}$ can be written uniquely as a product of proper $\ell$-radical elements $x_1\leq\cdots\leq x_n$.
\end{corollary}

\begin{proof} First observe that by Theorem~\ref{Thm 4.6}, each nonzero element of $\mathcal{L}$ has such a representation.
Let $x\ne 0$ be an element of $\mathcal{L}$. Suppose that $x=x_1\cdots x_n=y_1\cdots y_t$, where $x_1\leq\cdots\leq x_n$ and $y_1\leq\cdots\leq y_t$ are proper $\ell$-radical elements. Necessarily $x_1=\sqrt{x}=y_1$, so that since by Theorem~\ref{Thm 6.5}, $\sum_{i=1}^n\alpha_{x_i}=\alpha_x=\sum_{j=1}^t\alpha_{y_j}$, it follows that $\sum_{i=2}^n\alpha_{x_i}= \sum_{j=2}^t\alpha_{y_j}$. Since the mapping $\phi$ in Theorem~\ref{Thm 6.5} is injective, we have that $x_2\cdots x_n=y_2\cdots y_t$. Repeating the argument we obtain that $n=t$ and $x_i=y_i$ for all $i=1,2,\ldots,n$.
\end{proof}

\section{Applications to commutative rings}

We now interpret the results of the last sections in the context of commutative rings by viewing the set consisting of the regular ideals of a ring and the zero ideal of the ring as a principally generated lattice domain. By doing so, we obtain in Theorem~\ref{Thm 7.4} a characterization of a class of rings whose ideals are a product of radical ideals, followed by similar characterizations for domains in Corollary~\ref{Cor 7.7}. Because it comes at no extra expense of effort, we work more generally in the first two lemmas, where the focus is on a situation in which the role that the total quotient ring plays for regular ideals is replaced with a ring extension. The notion of regularity is replaced with a relativized notion that is flexible enough to cover both subclasses of regular ideals as well as ideals that need not contain a nonzerodivisor.

We recall several relevant definitions from
\cite{KZ}. If $R \subseteq T$ is an extension of commutative rings, an ideal $I$ of $R$ is {\it $T$-regular} if $IT = T$. For example, if $T$ is the total quotient ring of $R$, then an ideal is $T$-regular if and only if it is regular in the usual sense of containing a nonzerodivisor, while if $T$ is the complete ring of quotients of $R$, then an ideal is $T$-regular if and only if no nonzero element annihilates it.
An ideal $I$ of $R$ is {\it $T$-invertible} if there is an $R$-submodule $J$ of $T$ such that $IJ = R$. The extension $R \subseteq T$ is {\it tight} if for each $t \in T$ there is a $T$-invertible ideal $I$ with $tI\subseteq R$. Any ring of quotients is tight (see~\cite[p.~39]{KZ}).

\begin{lemma}\label{Lem 7.1} Let $R\subseteq T$ be an extension of commutative rings, and let $\mathcal{L}$ be the partially ordered set (ordered by inclusion) consisting of the $T$-regular ideals of $R$ and the zero ideal.
\begin{enumerate}
\item $\mathcal{L}$ is a $C$-lattice domain having the same residual $(I:J)$ as the lattice of all ideals of $R$.
\item The compact elements of $\mathcal{L}$ are the finitely generated $T$-regular ideals of $R$ together with the zero ideal.
\item If $R\subseteq T$ is a tight extension, then the nonzero $\ell$-principal elements of $\mathcal{L}$ are the $T$-invertible ideals of $R$ and are $\ell$-invertible elements of $\mathcal{L}$.
\end{enumerate}
\end{lemma}

\begin{proof} (1) and (2). Since the intersection of any two $T$-regular ideals $I$ and $J$ is $T$-regular (as it contains $IJ$), and the sum of any two $T$-regular ideals is $T$-regular, we may view $\mathcal{L}$ as a sublattice of the complete lattice of all ideals of $R$. It is clear that the join of an arbitrary subset of $\mathcal{L}$ is again in $\mathcal{L}$. If $\mathcal{F}$ is a subset of $\mathcal{L}$ such that $\bigcap_{I\in\mathcal{F}}I$ is not $T$-regular, then we set $\bigwedge\mathcal{F}=0$. Otherwise, we set $\bigwedge\mathcal{F}=\bigcap_{I\in\mathcal{F}}I$. With this definition of arbitrary meets, $\mathcal{L}$ is a multiplicative lattice (see \cite[pp.~409--410]{AP}). Moreover, as discussed in \cite[p.~410]{AP}, the residuation $(I:J)$ is the same whether defined relative to $\mathcal{L}$ or the lattice of all ideals of $R$. It follows that $\mathcal{L}$ is a multiplicative lattice. Since the nonzero elements in $\mathcal{L}$ are $T$-regular, $\mathcal{L}$ is a lattice domain.

To see next that $\mathcal{L}$ is a $C$-lattice, we first verify (2).
Let $I$ be a compact element in $\mathcal{L}$. Since $IT=T$, $I$ is the sum of the finitely generated $T$-regular ideals of $R$ contained in $I$. Compactness of $I$ now implies that $I$ is finitely generated. The converse, that a finitely generated $T$-regular ideal is compact in $\mathcal{L}$, is routine.
Since every $T$-regular ideal of $R$ is the sum of the finitely generated $T$-regular ideals contained in it, this proves that $\mathcal{L}$ is a $C$-lattice.

(3) Let $I$ be a nonzero $\ell$-principal element in $\mathcal{L}$. Since $I$ is a $T$-regular ideal and $R\subseteq T$ is tight, there is a $T$-invertible ideal $A$ contained in $I$. Since $I$ is weak meet principal, we have that $IJ=A$ for some $T$-regular ideal $J$ of $R$. There is some $R$-submodule $D$ of $T$ for which $AD=R$. We infer that $I(JD)=R$, proving that $I$ is $T$-invertible. This also implies $I$ is cancellative since if $IB\subseteq IC$ for ideals $B,C$ of $R$, then $(JD)IB\subseteq (JD)IC$, so that $B\subseteq C$. Conversely, if $I$ is a $T$-invertible ideal of $R$, it is routine to see that $I$ is $\ell$-principal in $\mathcal{L}$.
\end{proof}

In the next lemma we work under the assumption that every $T$-regular ideal of $R$ is a sum of $T$-invertible ideals. This can be viewed as a generalization of the Marot property that requires of a ring that every regular ideal is generated by regular elements. However, the former assumption is quite a bit broader than the Marot property, since for example any Pr\"ufer ring also satisfies it. We point this out again in Remark~\ref{Rem 7.6}.

Lemma~\ref{Lem 7.1} situates the lattice of $T$-regular ideals in the context of multiplicative lattices. With the lemma, we may apply Theorem~\ref{Thm 4.6} to obtain a characterization of the radical factorization property for $T$-regular ideals. If $R\subseteq T$ is an extension of rings, we denote by $\mbox{\rm Max}^{-1}_T(R)$ the set of maximal $T$-regular ideals with the {\it inverse topology}, that is, the topology having as a basis of open sets the sets of the form $V(I):=\{M\in\mbox{\rm Max}^{-1}_T(R)\mid I\subseteq M\}$, where $I$ is $T$-regular and finitely generated. The extension $R \subseteq T$ is a {\it Pr\"ufer extension} if it is tight and every finitely generated $T$-regular ideal is $T$-invertible.

\begin{lemma}\label{Lem 7.2} The following are equivalent for a tight extension $R\subseteq T$ such that every $T$-regular ideal of $R$ is a sum of $T$-invertible ideals.
\begin{enumerate}
\item Every $T$-regular ideal is a product of radical ideals.
\item Each $T$-regular ideal is a product of (unique proper) radical ideals $J_1\subseteq\cdots\subseteq J_n$.
\item Every $T$-regular prime ideal is maximal and each $T$-invertible ideal is a product of radical ideals.
\item Each $T$-regular prime ideal is maximal and contains a $T$-invertible radical ideal.
\item The radical of each finitely generated $T$-regular ideal is $T$-invertible.
\item The multiplicative lattice consisting of the $T$-regular ideals and the zero ideal is isomorphic to $U_b(\mbox{\rm Max}_T^{-1}(R))$.
\item $R\subseteq T$ is a Pr\"ufer extension for which the radical of each finitely generated $T$-regular ideal is finitely generated.
\end{enumerate}
\end{lemma}

\begin{proof} By Lemma~\ref{Lem 7.1}, the lattice $\mathcal{L}$ consisting of the $T$-regular ideals and the zero ideal is a $C$-lattice domain. The assumption that every $T$-regular ideal is a sum of $T$-invertible ideals implies that $\mathcal{L}$ is principally generated. Thus statements (1)--(7) follow from Lemma~\ref{Lem 7.1}, Theorems~\ref{Thm 4.6} and~\ref{Thm 6.5} and Corollary~\ref{Cor 6.7}.
\end{proof}

\begin{remark}\label{Rem 7.3} If $R \subseteq T$ is a Pr\"ufer extension, then it is clear that every $T$-regular ideal is a sum of $T$-invertible ideals. Thus statements (1)--(7) of Theorem~\ref{Lem 7.2} are equivalent under the lone hypothesis that $R \subseteq T$ is a Pr\"ufer extension.
\end{remark}

Specializing to the case where $T$ is the total quotient ring of $R$, we obtain the main theorem of this section, which generalizes to a larger class of rings some known characterizations of Marot SP-rings and adds a new one. Specifically, under the more restrictive assumption that $R$ is a Marot $N$-ring, the equivalence of (1)--(6) is proved by Ahmed and Dumitrescu in \cite[Theorem 2.12]{AD}. (A Marot ring is an $N$-ring if for each regular maximal ideal $M$ of $R$, $R_M$ is a discrete rank one Manis valuation ring.) In the theorem, we use $\mbox{\rm Max}_{\rm{reg}}^{-1}(R)$ to denote the set of maximal regular ideals of $R$ with respect to the inverse topology. Alternatively, this space can be viewed as $\mbox{\rm Max}_{Q(R)}^{-1}(R)$.

\begin{theorem}\label{Thm 7.4} The following are equivalent for a ring $R$ for which every regular ideal is a sum of invertible ideals.
\begin{enumerate}
\item $R$ is an SP-ring.
\item Each regular ideal is a product of (unique proper) radical ideals $J_1\subseteq\cdots\subseteq J_n$.
\item Every regular prime ideal is maximal and every invertible ideal of $R$ is a product of radical ideals of $R$.
\item Each regular prime ideal is maximal and contains an invertible radical ideal.
\item The radical of each regular finitely generated ideal is invertible.
\item The multiplicative lattice consisting of the regular ideals and the zero ideal is isomorphic to $U_b(\mbox{\rm Max}^{-1}_{\rm reg}(R))$.
\item $R$ is a Pr\"ufer ring for which the radical of each finitely generated regular ideal is finitely generated.
\end{enumerate}
\end{theorem}

\begin{proof} Apply Lemma~\ref{Lem 7.2} in the case in which $T$ is the total quotient ring of $R$.
\end{proof}

As in Remark~\ref{Rem 7.3}, a Pr\"ufer ring has the property that every regular ideal is a sum of invertible ideals, so we have the following corollary.

\begin{corollary}\label{Cor 7.5} Statements (1)--(7) of Theorem~\ref{Thm 7.4} are equivalent for a Pr\"ufer ring~$R$.
\end{corollary}

\begin{remark}\label{Rem 7.6} A Marot $N$-ring is a Pr\"ufer ring, so the equivalence of (1)--(6) in the corollary can also be viewed as a generalization of \cite[Theorem 2.12]{AD}.
\end{remark}

Specializing Theorem~\ref{Thm 7.4} to domains, we obtain the following characterization.

\begin{corollary}\label{Cor 7.7} The following are equivalent for an integral domain $R$.
\begin{enumerate}
\item Each proper ideal of $R$ is a product of radical ideals.
\item Each (nonzero) ideal is a product of (unique proper) radical ideals $J_1\subseteq\cdots\subseteq J_n$.
\item $\dim(R)\leq 1$ and each invertible ideal of $R$ is a product of radical ideals of $R$.
\item Each nonzero prime ideal is maximal and contains an invertible radical ideal.
\item The radical of each nonzero finitely generated ideal is invertible.
\item The multiplicative lattice of ideals of $R$ is isomorphic to $U_b(\mbox{\rm Max}^{-1}(R))$.
\item $R$ is a Pr\"ufer domain for which the radical of each finitely generated ideal is finitely generated.
\end{enumerate}
\end{corollary}

The equivalence of (1) and (2) can also be found in \cite[Lemma 4.2]{HOR} (see also \cite[Theorem 2.1]{Olb}). The rings satisfying (1) are known as $SP$-domains in the literature (where ``SP'' stands for semi-prime). See \cite{Olb} for background on this class rings, and see also \cite{AD,FHL,HOR, MC, R} for more characterizations and properties of these rings.

We mention another consequence of Lemma~\ref{Lem 7.2}, this one concerned with a ``neighborhood'' version of radical factorization that shows that invertible zero-dimensional radical ideals give rise to ideals that factor into radical ideals.

\begin{theorem}\label{Thm 7.8} Let $R$ be a commutative ring. If $J$ is an invertible radical zero-dimensional ideal, then every ideal that contains a power of $J$ can be written uniquely as a product of proper radical ideals $J_1\subseteq\cdots\subseteq J_n$.
\end{theorem}

\begin{proof} Let $Q(R)$ be the total quotient ring of $R$, and let $T=\bigcup_{k\in\mathbb{N}_0} (R:_{Q(R)}J^k)$. Then $R\subseteq T$ is a tight extension for which every $T$-regular ideal is a sum of invertible $T$-ideals. The multiplicative lattice consisting of the zero ideal and the ideals containing a power of $J$ is the lattice of $T$-regular ideals and the zero ideal. The claim now follows from Lemma~\ref{Lem 7.2}.
\end{proof}

\begin{remark}\label{Rem 7.9} Under the hypotheses of Theorem~\ref{Thm 7.8}, we obtain also from Lemma~\ref{Lem 7.2} that the multiplicative lattice $\mathcal{L}$ consisting of the zero ideal and the ideals that contain a power of $J$ is isomorphic to $U_b(\mbox{\rm Max}_T^{-1}(R))$. Since the maximal $T$-regular ideals are the maximal ideals of $R$ containing $J$, it follows that $\mbox{\rm Max}_T^{-1}(R)$ and $\mbox{\rm Max}^{-1}(R/J)$ are homeomorphic. Since $\dim(R/J)=0$, the Zariski and inverse topologies agree on $\mbox{\rm Max}(R/J)$ \cite[Lemma 6.3]{HOR}. Therefore, $\mbox{\rm Max}^{-1}(R/J)=\mbox{\rm Max}(R/J)$, and hence $\mathcal{L}\cong U_b(\mbox{\rm Max}(R/J)).$
\end{remark}

\begin{remark}\label{Rem 7.10} Ahmed and Dumitrescu \cite{AD} define a ring $R$ to be an {\it SSP-ring} if every ideal of $R$ is a product of radical ideals. Remark~\ref{Rem 4.5} generalizes a theorem of Ahmed and Dumitrescu \cite{AD} that states an SSP-ring is an almost multiplication ring, i.e., a ring $R$ such that $R_M$ is a discrete valuation domain or a special principal ideal ring for all $M\in\mbox{\rm Max}(R)$.
\end{remark}

\section{Ideal systems and star operations}

We next apply the results of Section 2--6 to ideal systems of monoids, which allows us to give characterizations of the radical factorization property for subclasses of ideals subject to natural closure conditions. At the end of this section, we translate these characterizations into the context of modular star operations on integral domains.

\begin{center}
\textit{Throughout this section let $H$ be a commutative multiplicative monoid.}
\end{center}

By $z(H)$ we denote the set of zero elements of $H$ (i.e., the set of elements $z\in H$ for which $xz=z$ for each $x\in H$) and by $\mathbb{P}(H)$ we denote the power set of $H$. We say that $H$ is cancellative if every $x\in H\setminus z(H)$ is cancellative. Let $r:\mathbb{P}(H)\rightarrow\mathbb{P}(H)$, $X\mapsto X_r$ be a map. We say that $r$ is a weak ideal system on $H$ if $r$ satisfies the following properties for all $X,Y\subseteq H$ and $c\in H$.
\begin{itemize}
\item[(A)] $XH\cup z(H)\subseteq X_r$.
\item[(B)] If $X\subseteq Y_r$, then $X_r\subseteq Y_r$.
\item[(C)] $cX_r\subseteq (cX)_r$.
\end{itemize}

Let $r$ be a weak ideal system on $H$. We say that $r$ is an ideal system on $H$ if $cX_r=(cX)_r$ for all $X\subseteq H$ and $c\in H$. Moreover, $r$ is called finitary if $X_r=\bigcup_{E\subseteq X,|E|<\infty} E_r$ for all $X\subseteq H$ (equivalently, $X_r\subseteq\bigcup_{E\subseteq X,|E|<\infty} E_r$ for all $X\subseteq H$). We say that a subset $X\subseteq H$ is an $r$-ideal of $H$ if $X_r=X$. Moreover, $X_r$ is called the $r$-closure of $X$.

A cancellative element $x\in H$ is called regular if $(xA)_r=xA_r$ for all $A\subseteq H$. We say that
\begin{enumerate}
\item[(1)] $I$ is proper if $I\subsetneq H$.
\item[(2)] $I$ is nontrivial if $z(H)\subsetneq I$.
\item[(3)] $I$ is regular if it contains a regular element of $H$.
\item[(4)] $I$ is $r$-invertible if $(IJ)_r=yH$ for some $J\subseteq H$ and some regular $y\in H$.
\end{enumerate}
By $\mathcal{I}_r(H)$ (resp. $\mathcal{I}_r(H)_{\rm reg}$) we denote the set of $r$-ideals (resp. the set of regular $r$-ideals together with the bottom element $\emptyset_r$) of $H$. Let $I$ be an $r$-ideal of $H$. We say that $I$ is $r$-finitely generated if $I=E_r$ for some finite $E\subseteq I$. Furthermore, we say that $r$ is modular if for all $r$-ideals $I,J,N$ of $H$ with $I\subseteq N$ it follows that $(I\cup J)_r\cap N\subseteq (I\cup (J\cap N))_r$ (equivalently: for all $r$-ideals $I,J,N$ of $H$ with $I\subseteq N$ it follows that $(I\cup J)_r\cap N=(I\cup (J\cap N))_r$).

Next we introduce two important ideal systems, namely the $s$-system and the $d$-system. Let $R$ be a ring.

Let $$s:\mathbb{P}(H)\rightarrow\mathbb{P}(H), X\mapsto XH\cup z(H).$$

Let $$d:\mathbb{P}(R)\rightarrow\mathbb{P}(R),X\mapsto {}_R(X),$$ where is ${}_R(X)$ is the (ring) ideal of $R$ generated by $X$. It is straightforward to prove that $s$ resp. $d$ are modular finitary ideal systems on $H$ resp. $R$.

Let $r$ be a weak ideal system on $H$. We define the $r$-multiplication $\mathcal{I}_r(H)\times\mathcal{I}_r(H)\rightarrow\mathcal{I}_r(H)$ by $(I,J)\mapsto (IJ)_r$. Then $\mathcal{I}_r(H)$ is a multiplicative lattice, where the multiplication is the $r$-multiplication and the partial order is inclusion (see \cite[Chapter 8]{HA}). Also note that if $\mathcal{J}\subseteq\mathcal{I}_r(H)$, then $\bigvee\mathcal{J}=(\bigcup_{J\in\mathcal{J}} J)_r$ and $\bigwedge\mathcal{J}=\bigcap_{J\in\mathcal{J}} J$. Moreover, $\mathcal{I}_r(H)_{\rm reg}$ forms a multiplicative lattice domain under restricted $r$-multiplication and inclusion. (Clearly, finite $r$-products and arbitrary joins of regular $r$-ideals are regular. The meet of elements in $\mathcal{I}_r(H)_{\rm reg}$ is the meet of these elements in $\mathcal{I}_r(H)$ if it is regular and the bottom element otherwise.)

\begin{lemma}\label{Lem 8.1} Let $r$ be a weak ideal system on $H$. Each compact element of $\mathcal{I}_r(H)$ and each compact element of $\mathcal{I}_r(H)_{\rm reg}$ is $r$-finitely generated and the following are equivalent:
\begin{enumerate}
\item[(A)] $r$ is finitary.
\item[(B)] Every $r$-finitely generated $r$-ideal of $H$ is a compact element of $\mathcal{I}_r(H)$.
\item[(C)] $\{x\}_r$ is a compact element of $\mathcal{I}_r(H)$ for every $x\in H$.
\end{enumerate}
If these equivalent conditions are satisfied, then $\mathcal{I}_r(H)$ and $\mathcal{I}_r(H)_{\rm reg}$ are both $C$-lattices.
\end{lemma}

\begin{proof} Let $I\in\mathcal{I}_r(H)$ be compact. Then $I=I_r=(\bigcup_{x\in I}\{x\}_r)_r=\bigvee\{\{x\}_r\mid x\in I\}$, and hence there is some finite $E\subseteq I$ such that $I=\bigvee\{\{x\}_r\mid x\in E\}$. Therefore, $I=E_r$ is $r$-finitely generated.

Now let $A\in\mathcal{I}_r(H)_{\rm reg}$ be nontrivial and compact and let $z\in A$ be regular. Since $A=\bigvee\{\{x,z\}_r\mid x\in A\}$, we have that $A=\bigvee\{\{x,z\}_r\mid x\in F\}$ for some finite $F\subseteq A$. Therefore, $A=(F\cup\{z\})_r$ is $r$-finitely generated.

(A) $\Rightarrow$ (B) Let $r$ be finitary, $I$ an $r$-finitely generated $r$-ideal of $H$ and $\mathcal{E}$ a set of $r$-ideals of $H$ such that $I\leq\bigvee\mathcal{E}$. There is some finite $F\subseteq I$ such that $I=F_r$. We have that $F\subseteq (\bigcup_{J\in\mathcal{E}} J)_r$. Since $r$ is finitary, there is some finite $F^{\prime}\subseteq\bigcup_{J\in\mathcal{E}} J$ such that $F\subseteq (F^{\prime})_r$. Clearly, there is some finite $\mathcal{E}^{\prime}\subseteq\mathcal{E}$ such that $F^{\prime}\subseteq\bigcup_{J\in\mathcal{E}^{\prime}} J$, and thus $I=F_r\subseteq F^{\prime}\subseteq (\bigcup_{J\in\mathcal{E}^{\prime}} J)_r$. We infer that $I\leq\bigvee\mathcal{E}^{\prime}$, and hence $I$ is compact.

(B) $\Rightarrow$ (C) Trivial.

(C) $\Rightarrow$ (A) Let $X\subseteq H$ and $x\in X$. Since $\{x\}_r$ is compact and $\{x\}_r\leq X_r=\bigvee\{\{y\}_r\mid y\in X\}$, there is some finite $E\subseteq X$ such that $\{x\}_r\leq\bigvee\{\{y\}_r\mid y\in E\}$ and hence $x\in\{x\}_r\subseteq (\bigcup_{y\in E}\{y\}_r)_r=E_r$.

Now let the above conditions be satisfied. We infer that the set of compact elements of $\mathcal{I}_r(H)$ is the set of $r$-finitely generated $r$-ideals of $H$, and hence it is multiplicatively closed. Since $J=\bigvee\{\{x\}_r\mid x\in J\}$ for each $J\in\mathcal{I}_r(H)$, we infer that $\mathcal{I}_r(H)$ is a $C$-lattice.

Clearly, every compact element of $\mathcal{I}_r(H)$ that is an element of $\mathcal{I}_r(H)_{\rm reg}$ is a compact element of $\mathcal{I}_r(H)_{\rm reg}$. We infer that the set of compact elements of $\mathcal{I}_r(H)_{\rm reg}$ is the set of regular $r$-finitely generated $r$-ideals of $H$ together with the bottom element, and thus it is multiplicatively closed. If $J\in\mathcal{I}_r(H)_{\rm reg}$ and $y\in J$ is regular, then since $J=\bigvee\{\{x,y\}_r\mid x\in J\}$ is a join of compact elements of $\mathcal{I}_r(H)_{\rm reg}$, we have that $\mathcal{I}_r(H)_{\rm reg}$ is a $C$-lattice.
\end{proof}

\begin{lemma}\label{Lem 8.2} Let $r$ be a weak ideal system on $H$.
\begin{enumerate}
\item[(1)] Every $r$-invertible $r$-ideal of $H$ is a weak meet principal and cancellative element of both $\mathcal{I}_r(H)$ and $\mathcal{I}_r(H)_{\rm reg}$.
\item[(2)] If $r$ is modular, then every $r$-invertible $r$-ideal of $H$ is an $\ell$-invertible element of both $\mathcal{I}_r(H)$ and $\mathcal{I}_r(H)_{\rm reg}$.
\item[(3)] Every $\ell$-invertible element of $\mathcal{I}_r(H)_{\rm reg}$ is $r$-invertible.
\end{enumerate}
\end{lemma}

\begin{proof} (1) Let $I$ be an $r$-invertible $r$-ideal of $H$. It is clear that $I\in\mathcal{I}_r(H)_{\rm reg}$. Therefore, it is sufficient to show that $I$ is a weak meet principal and cancellative element of $\mathcal{I}_r(H)$. There are some $B\in\mathcal{I}_r(H)$ and some regular $y\in H$ such that $(IB)_r=yH$.

First we show that $I$ is a cancellative element of $\mathcal{I}_r(H)$. Let $J,L\in\mathcal{I}_r(H)$ be such that $(IJ)_r=(IL)_r$. Then $yJ=(yJ)_r=(IBJ)_r=((IJ)_rB)_r=((IL)_rB)_r=(IBL)_r=(yL)_r=yL$, and thus $J=L$.

Next we show that $I$ is weak meet principal. It is sufficient to show that for each $J\in\mathcal{I}_r(H)$ such that $J\subseteq I$, there is some $A\in\mathcal{I}_r(H)$ such that $J=(AI)_r$. Let $J\in\mathcal{I}_r(H)$ be such that $J\subseteq I$. Then $(JB)_r\subseteq (IB)_r=yH$. Set $A=(\{z\in H\mid yz\in (JB)_r\})_r$. Then $A\in\mathcal{I}_r(H)$. Let $x\in (JB)_r$. Then $x=yz$ for some $z\in H$. We have that $z\in A$, and thus $x\in yA$. Moreover, $yA=(y\{z\in H\mid yz\in (JB)_r\})_r\subseteq (JB)_r$. Therefore, $(JB)_r=yA$, and thus $yJ=(yHJ)_r=((IB)_rJ)_r=(IBJ)_r=(JBI)_r=((JB)_rI)_r=(yAI)_r=y(AI)_r$. We infer that $J=(AI)_r$.

(2) This is an immediate consequence of (1) and Lemma~\ref{Lem 2.10}(3).

(3) Let $I$ be an $\ell$-invertible element of $\mathcal{I}_r(H)_{\rm reg}$. There is some regular $x\in I$. Since $xH\subseteq I$ and $I$ is weak meet principal, it follows that $xH=xH\cap I=((xH:I)I)_r$. On the other hand, $(xH:I)\in\mathcal{I}_r(H)$, and thus $I$ is $r$-invertible.
\end{proof}

\begin{corollary}\label{Cor 8.3} Let $r$ be a modular weak ideal system on $H$ such that every regular $r$-ideal of $H$ is an $r$-union of $r$-invertible $r$-ideals of $H$. Then $\mathcal{I}_r(H)_{\rm reg}$ is a principally generated lattice.
\end{corollary}

\begin{proof} We infer by Lemma \ref{Lem 8.2}(2) that every $r$-invertible $r$-ideal of $H$ is an $\ell$-invertible element of $\mathcal{I}_r(H)_{\rm reg}$. Consequently, $\mathcal{I}_r(H)_{\rm reg}$ is principally generated, since the $r$-union of elements in $\mathcal{I}_r(H)_{\rm reg}$ is their join in $\mathcal{I}_r(H)_{\rm reg}$.
\end{proof}

Let $r$ be a weak ideal system on $H$. We say that $H$ is an $r$-Pr\"ufer monoid if every regular $r$-finitely generated $r$-ideal is $r$-invertible. Note that every $r$-Pr\"ufer monoid satisfies the condition that every regular $r$-ideal of $H$ is an $r$-union of $r$-invertible $r$-ideals in Corollary~\ref{Cor 8.3}. Moreover, note that this condition is satisfied by every $r$-Marot monoid (i.e., a monoid for which every regular $r$-ideal is the $r$-closure of a set of regular elements). An $r$-ideal $I$ of $H$ is called radical if for each $n\in\mathbb{N}$ and $x\in H$ such that $x^n\in I$ it follows that $x\in I$. Moreover, $H$ is called an $r$-SP-monoid if every $r$-ideal is a finite $r$-product of radical $r$-ideals of $H$.

Furthermore, let $r$-$\max(H)$ denote the set of $r$-maximal $r$-ideals (resp. let $r$-$\max(H)_{\rm reg}$ denote the set of regular $r$-maximal $r$-ideals of $H$). Let $r$-$\max^{-1}(H)$ resp. $r$-$\max^{-1}(H)_{\rm reg}$ be the corresponding topological spaces equipped with the inverse topology (as defined in Definition~\ref{Def 6.1}). Moreover, set $\dim_r(H)=\dim\mathcal{I}_r(H)$.

\begin{theorem}\label{Thm 8.4} The following are equivalent for a modular finitary weak ideal system $r$ on $H$ such that each regular $r$-ideal of $H$ is an $r$-union of $r$-invertible $r$-ideals of $H$.
\begin{enumerate}
\item Every regular $r$-ideal of $H$ is an $r$-product of radical $r$-ideals.
\item Every regular prime $r$-ideal is $r$-maximal and each $r$-invertible $r$-ideal is an $r$-product of radical $r$-ideals.
\item Each regular prime $r$-ideal is $r$-maximal and contains an $r$-invertible radical $r$-ideal.
\item Each regular $r$-ideal is an $r$-product of (unique proper) radical $r$-ideals $J_1\subseteq\cdots\subseteq J_n$.
\item The radical of each regular $r$-finitely generated $r$-ideal is $r$-invertible.
\item The $C$-lattice $\mathcal{I}_r(H)_{\rm reg}$ is isomorphic to $U_b(r$-$\max^{-1}(H)_{\rm reg})$.
\item $H$ is an $r$-Pr\"ufer monoid and the radical of every $r$-finitely generated $r$-ideal is $r$-finitely generated.
\end{enumerate}
\end{theorem}

\begin{proof} This is an immediate consequence of Theorems~\ref{Thm 4.6}, \ref{Thm 5.3} and \ref{Thm 6.5}, Corollaries~\ref{Cor 6.7} and \ref{Cor 8.3} and Lemma~\ref{Lem 8.1}.
\end{proof}

\begin{corollary}\label{Cor 8.5} The following are equivalent if $H$ is cancellative and $r$ is a modular finitary ideal system on $H$.
\begin{enumerate}
\item $H$ is an $r$-SP-monoid.
\item $\dim_r(H)\leq 1$ and each $r$-invertible $r$-ideal is an $r$-product of radical $r$-ideals.
\item $\dim_r(H)\leq 1$ and each nontrivial prime $r$-ideal contains an $r$-invertible radical $r$-ideal.
\item Each (nontrivial) $r$-ideal is an $r$-product of (unique proper) radical $r$-ideals $J_1\subseteq\cdots\subseteq J_n$.
\item The radical of each nontrivial $r$-finitely generated $r$-ideal is $r$-invertible.
\item The $C$-lattice $\mathcal{I}_r(H)$ is isomorphic to $U_b(r$-$\max^{-1}(H))$.
\item $H$ is an $r$-Pr\"ufer monoid and the radical of every $r$-finitely generated $r$-ideal is $r$-finitely generated.
\end{enumerate}
\end{corollary}

\begin{proof} Observe that $xH$ is an $r$-invertible $r$-ideal of $H$ for every cancellative element $x\in H$. Let $I$ be a nontrivial $r$-ideal of $H$. Clearly, $I$ is regular and $I=(\bigcup_{x\in I\setminus z(H)} xH)_r$. The statement now follows by Theorem \ref{Thm 8.4}.
\end{proof}

Let $R$ be an integral domain. An ideal system $r$ on $R$ is called a star operation on $R$ if every $r$-ideal of $R$ is a (ring) ideal of $R$. A star operation on $R$ is called of finite type if it is finitary (as an ideal system). A star operation is called modular if it is modular as an ideal system. Note that this concept of star operation differs from the classical notion of star operation, but it leads to the same monoid of star ideals as the ``classical star operations''. More precisely, if $r$ is a star operation on $R$ and $\mathcal{F}(R)$ denotes the set of nonzero fractional ideals of $R$, then $*:\mathcal{F}(R)\rightarrow\mathcal{F}(R)$ defined by $X_*=c^{-1}(cX)_r$ for all $X\in\mathcal{F}(R)$ and nonzero $c\in R$ such that $cX\subseteq R$ is a ``classical star operation'' on $R$ and $\mathcal{I}_r(R)=\{I\in\mathcal{F}(R)\mid I\subseteq R,I_*=I\}\cup\{(0)\}$. Conversely if $*:\mathcal{F}(R)\rightarrow\mathcal{F}(R)$ is a ``classical star operation'' on $R$, then $r:\mathbb{P}(R)\rightarrow\mathbb{P}(R)$ defined by $X_r=({}_R(X))_*$ if $X\nsubseteq\{0\}$ and by $X_r=\{0\}$ if $X\subseteq\{0\}$, then $r$ is a star operation on $R$. If $*$ is a star operation of finite type, then we say that $R$ is a P$*$MD if every nonzero $*$-finitely generated $*$-ideal of $R$ is $*$-invertible. We say that $R$ is a $*$-SP-domain if every $*$-ideal of $R$ a finite $*$-product of radical $*$-ideals of $R$.

\begin{corollary}\label{Cor 8.6} The following are equivalent for an integral domain $R$ and a modular star operation $*$ on $R$ of finite type.
\begin{enumerate}
\item $R$ is a $*$-SP-domain.
\item $\dim_*(R)\leq 1$ and every $*$-invertible $*$-ideal of $R$ is a finite $*$-product of radical $*$-ideals of $R$.
\item Each nonzero prime $*$-ideal is maximal and contains a $*$-invertible radical $*$-ideal.
\item Each $*$-ideal is a product of (unique) radical $*$-ideals $J_1\subseteq\cdots\subseteq J_n$.
\item The radical of each nonzero $*$-finitely generated $*$-ideal is $*$-invertible.
\item The $C$-lattice $\mathcal{I}_*(R)$ is isomorphic to $U_b(*$-$\max^{-1}(R))$.
\item $R$ is a P$*$MD for which the radical of each $*$-finitely generated $*$-ideal is $*$-finitely generated.
\end{enumerate}
\end{corollary}

\begin{proof} This follows from Corollary~\ref{Cor 8.5}.
\end{proof}

\subsection*{Acknowledgements} We would like to thank T. Dumitrescu for his careful reading and for many useful comments and remarks that helped to simplify and shorten several proofs of this paper. We also want to thank the referee for his/her suggestions and comments.

\end{document}